\documentclass[12pt]{amsart}
\usepackage[margin=3cm]{geometry}               
\usepackage[usenames,dvipsnames,svgnames,table]{xcolor}    
\usepackage{graphicx}
\usepackage{lineno, hyperref}
\usepackage{amssymb,amsfonts, amsmath, amsthm}
\usepackage{thmtools}
\usepackage{epstopdf}
\usepackage[shortlabels]{enumitem}
\usepackage{cancel}
\usepackage{thmtools} 
\usepackage{mathtools} 
\usepackage{mathrsfs}
\usepackage{yhmath}
\usepackage[english]{babel}
\usepackage[toc,page]{appendix}
\usepackage{float}
\restylefloat{table}
\usepackage{lineno}
\usepackage{enumitem}
\usepackage{chngcntr}
\usepackage{xcolor}

 \usepackage{hyperref}
  \hypersetup{
colorlinks=true,
linkcolor=cyan,
citecolor=cyan
}

\author[Cubides Kovacsics]{Pablo {Cubides Kovacsics}}
\address{Mathematisches Institut der Heinrich-Heine-Universit\"at D\"usseldorf, 
Universit\"atsstr. 1, 40225 D\"usseldorf, Germany. }
\email{cubidesk@hhu.de}
\author[Point]{Fran\c coise {Point}$^{(\dagger)}$}
\address{Department of Mathematics (De Vinci)\\ UMons\\ 20, place du Parc 7000 Mons, Belgium}
\email{point@math.univ-paris-diderot.fr}

\title[Topological fields with a generic derivation]{Topological fields with a generic derivation}

\swapnumbers

\newtheorem*{theorem*}{Theorem}
\newtheorem*{question*}{Question}
\newtheorem*{corollary*}{Corollary}

\numberwithin{equation}{subsection}

\newtheorem{theorem}[equation]{Theorem}
\newtheorem{lemma}[equation]{Lemma}
\newtheorem{fact}[equation]{Fact}
\newtheorem{claim}[equation]{Claim}
\newtheorem{proposition}[equation]{Proposition}
\newtheorem{corollary}[equation]{Corollary}

\theoremstyle{definition}
\newtheorem{definition}[theorem]{Definition}
\newtheorem{lemma-definition}[equation]{Lemma-Definition}

\newtheorem{remark}[equation]{Remark}

\newtheorem{examples}[equation]{Examples}

\newtheorem{question}[equation]{Question}
\newtheorem{notation}[equation]{Notation}

\newcommand{\G}{\mathcal G}

\newcommand{\R}{\mathbb R}
\newcommand{\Q}{\mathbb Q}
\newcommand{\N}{\mathbb N}
\newcommand{\Z}{\mathbb Z}

\newcommand{\cF}{\mathcal F}

\newcommand{\cM}{\mathcal M}

\newcommand{\bU}{\mathbb{U}}

\newcommand{\cL}{\mathcal{L}}

\newcommand{\cK}{\mathcal{K}}
\newcommand{\Div}{\mathrm{div}}

\newcommand{\Dim}{\mathrm{dim}}

\newcommand{\eq}{\mathrm{eq}}

\newcommand{\acl}{\mathrm{acl}}
\newcommand{\alg}{\mathrm{alg}}
\newcommand{\NIP}{\mathrm{NIP}}

\newcommand{\RCF}{\mathrm{RCF}}

\newcommand{\ACVF}{\mathrm{ACVF}}

\newcommand{\RCVF}{\mathrm{RCVF}}
\newcommand{\RV}{\mathrm{RV}}
\newcommand{\PCF}{p\mathrm{CF}}
\newcommand{\CODF}{\mathrm{CODF}}

\newcommand{\Th}{\mathrm{Th}}

\newcommand{\Graph}{\mathrm{graph}}

\newcommand{\Mor}{\mathrm{Mor}}
\newcommand{\lex}{\mathrm{lex}}

\newcommand{\id}{\mathrm{id}}

\newcommand{\J}{\nabla}

\newcommand{\cA}{\mathcal{A}}
\newcommand{\cB}{\mathcal{B}}

\newcommand{\cZ}{\mathcal{Z}}

\newcommand{\ord}{\mathrm{ord}}
\newcommand{\Int}{\mathrm{Int}}

\newcommand{\bd}{\bar{\delta}}

\newcommand{\bS}{\mathbf{S}}
\newcommand{\bF}{\mathbf{F}}

\begin{document}

\thanks{($\dagger$) Research director at the Fonds National de la Recherche Scientifique (FNRS-FRS)}

\begin{abstract} We study a class of tame $\cL$-theories $T$ of topological fields and their $\cL_\delta$-extension $T_{\delta}^*$ by a generic derivation $\delta$. The topological fields under consideration include henselian valued fields of characteristic 0 and real closed fields. We show that the associated expansion by a generic derivation has $\cL$-open core (i.e., every $\cL_\delta$-definable open set is $\cL$-definable) and derive both a cell decomposition theorem and a transfer result of elimination of imaginaries. Other tame properties of $T$ such as relative elimination of field sort quantifiers, $\NIP$ and distality also transfer to $T_\delta^*$. As an application, we derive consequences for the corresponding theories of dense pairs. In particular, we show that the theory of pairs of real closed fields (resp. of $p$-adically closed fields and real closed valued fields) admits a distal expansion. This gives a partial answer to a question of P. Simon.  
\end{abstract}

\maketitle

\setcounter{tocdepth}{1}

{
  \hypersetup{linkcolor=black}
  \tableofcontents
}

\section*{Introduction}

The study of topological fields with a derivation has been traditionally divided in two main branches. The first branch, as studied (in chronological order) in \cite{R1980, Scanlon, ADH, rideau}, treats the case where some compatibility between the derivation and the topology is assumed (\emph{e.g.}, continuity). The second branch, as studied in \cite{singer1978, Tressl, guzy-point2010, point2011, GP12, fornasiero-kaplan}, deals with the case where no such compatibility is required but rather a \emph{generic} behaviour of the derivation occurs. An example of such a generic behaviour arises in existentially closed ordered differential fields, a class studied and axiomatized by M. Singer in \cite{singer1978}. Each branch seems to tackle different aspects of differential fields and has its own applications. 

The purpose of this article is to further develop the study of generic derivations and show that many tame properties of theories of topological fields transfer to their expansions by such derivations. Examples of the topological fields under consideration include real closed fields and henselian valued fields of characteristic 0. We adopt a uniform treatment and development of such topological fields in the spirit of L. Mathews \cite{M} and A. Pillay \cite{pillay87}, which we consider interesting on its own. As an application of generic derivations, we derive consequences for the corresponding theories of dense pairs of topological fields (as studied in \cite{Robinson, Macintyre, Dries1998, berenstein-vassiliev2010, F}, to mention a few), supporting the idea that this framework is a useful tool to study such pairs of structures. 

The following section gathers a detailed overview of our main results.

\section*{Main results}\label{sec:mainresults}

In the first section of this article we introduce the theories of topological fields we will be concerned with, which we call \emph{open theories of topological fields}. Informally, an open theory of topological fields is a first order topological theory of fields in the sense of A. Pillay \cite{pillay87} (i.e., the topology is uniformly definable) in which definable sets are finite boolean combinations of Zariski closed sets and open sets. However, in contrast to \cite{pillay87}, our setting explicitly allows multi-sorted languages. Examples include complete theories of henselian valued fields of characteristic 0 and the theory of real closed fields. The formal definition will be given in Section \ref{sec:topopen}. 

The main focus of the present article is to study expansions of open theories of topological fields by \emph{generic} derivations. Let $T$ be an open $\cL$-theory of topological fields and $T_\delta$ be the theory $T$ together with axioms stating that $\delta$ is a derivation on the field sort (in the extension $\cL_{\delta}$ of $\cL$ by $\delta$). We define an $\cL_\delta$-extension $T_\delta^*$ of $T_\delta$ informally as follows. Models of $T_\delta^*$  satisfy that, for any unary differential polynomial $P$, if the ordinary polynomial associated with $P$ has a regular solution $a$, then one can find differential solutions of $P$ arbitrarily close to $a$. A derivation $\delta$ on a model $K$ of $T$ is \emph{generic} if $(K,\delta)$ is a model of $T_\delta^*$. 

When $T$ is the theory of real closed fields, the theory $T_\delta^*$ corresponds to the theory of closed ordered differential fields $\CODF$ as originally introduced and axiomatized by M. Singer in \cite{singer1978}. The idea behind $\CODF$ has been generalized to many different contexts including work by M. Tressl \cite{Tressl} and N. Solancki \cite{solanki} in the framework of large fields, and by N. Guzy and the second author in \cite{guzy-point2010,GP12}. As in \cite{guzy-point2010,GP12}, we will closely follow Singer's original axiomatization. The main difference in the present setting with respect to previous work is the explicit allowance of multi-sorted languages (which permits us to include theories in languages where they admit relative quantifier elimination).

How much complexity is introduced by a generic derivation? Two of our main contributions show that the derivation introduces no new complexity concerning both open definable sets and imaginaries. 

\begin{theorem*}[Later Theorem \ref{thm:opencoregen} and Corollary \ref{cor:opencore}] Let $T$ be an open $\cL$-theory of topological fields. Then, the theory $T_\delta^*$ has $\cL$-open core (i.e., every open $\cL_\delta$-definable set is already $\cL$-definable). In particular, when $T$ is either $\ACVF_{0,p}$ (with $p\geqslant 0$), $\RCVF$, $\PCF$ or, in general, the $\cL_\Div$-theory of a henselian valued field of characteristic 0, then, the theory $T_\delta^*$ has $\cL$-open core.
\end{theorem*}      
     
\begin{theorem*}[Later Theorem \ref{thm:EI}] Let $T$ be an open $\cL$-theory of topological fields. Let $\G$ be a collection of sorts of $\cL^{eq}$ and $\cL^{\G}$ denote the restriction of $\cL^{eq}$ to the sorts in $\G$. 
Suppose that $T$ admits elimination of imaginaries in $\cL^\G$. Then the theory $T^*_{\delta}$ admits elimination of imaginaries in $\cL_{\delta}^\G$. 
\end{theorem*}  

The previous theorem yields another proof of the fact that $\CODF$ has elimination of imaginaries (the first proof was given by the second author in \cite{point2011}, a second proof was given in \cite{BCP}). As a corollary we also obtain the following.

\begin{corollary*}[Later Corollary \ref{cor:EI_examples}] Let $\cL^\G$ denote the geometric language of valued fields as defined in \cite{HHM2006}. The theories $(\ACVF_{0,p})_\delta^*$, $\RCVF_\delta^*$ and $\PCF_\delta^*$ have elimination of imaginaries in $\cL_\delta^\G$. 
\end{corollary*} 

It is worth mentioning that the $\cL$-open core plays a crucial role to show the transfer of elimination of imaginaries. 

In return, the proof of the $\cL$-open core has two main ingredients. The first ingredient is a cell decomposition theorem for definable subsets of models of $T$ (see later Theorem \ref{thm:newCD}) in the spirit of results by P. Simon and E. Walsberg \cite{simon-walsberg2016} for dp-minimal topological structures. As in \cite{simon-walsberg2016}, cells are graphs of continuous correspondences (``multi-valued functions'') but, in addition, our definition of cell ensures that the projection of a cell (onto an initial subset of coordinates) is again a cell (see Definition \ref{def:cells}). The second ingredient is a parametric version of the density of differential elements in open sets (see Lemmas \ref{fact:density}, \ref{lem:envelop}).

Building on the notion of cell given for definable sets of models of $T$, we introduce a notion of cells for definable sets of models of $T_\delta^*$ which we call $\delta$-cells (see Definition \ref{def:delta-cell}). Informally, $X$ is a $\delta$-cell if there is an $\cL$-definable cell $Y$ such that $X$ is the pullback under (iterations of) $\delta$ of $Y$ and the differential prolongation of $X$ is dense in $Y$. Using both the $\cL$-open core and cell decomposition (in $T$), we prove a $\delta$-cell decomposition theorem for $T_\delta^*$. In the case of $\CODF$, it improves results by T. Brihaye, C. Michaux and C. Rivi\`ere in \cite{BMR} (see Remark \ref{rem:delta-CD}). 

\begin{theorem*}[Later Theorem \ref{delta-cell-decom}] Let $T$ be an open $\cL$-theory of topological fields and $K$ be a model of $T_\delta^*$. Every $\cL_\delta$-definable subset $X\subseteq K^n$ can be partitioned into finitely many $\delta$-cells. 
\end{theorem*}

Last but not least, we illustrate how the theory $T_{\delta}^*$ provides a useful setting to study the theory $T_P$ of dense pairs of models of a one-sorted open $\cL$-theory of topological fields $T$. We show that $T_{P}$ has $\cL$-open core (Theorem \ref{thm:opencorepairs}) and deduce that $T_P$ admits a distal expansion (namely $T_\delta^*$) whenever $T$ is a distal theory (Corollary \ref{cor:dist_expansion}). In particular, we show that $T_P$ admits a distal expansion when $T$ is $\RCF$, $\PCF$ and $\RCVF$. Our result gives a positive answer to a particular case of a question of P. Simon who asked if the theory of dense pairs of an o-minimal structure (extending a group) has a distal expansion (see \cite{nell2018} for a discussion). T. Nell provided a positive answer in the case of ordered vector fields \cite{nell2018}. Our result extends to pairs of real closed fields. Recently, A. Fornasiero and E. Kaplan \cite{fornasiero-kaplan} extended this result to other o-minimal expansions of real closed fields. 

\medskip

The paper is laid out as follows. Open theories of topological fields are studied in Section \ref{sec:topopen}. Correspondences are studied in \eqref{sec:correspon} and the cell decomposition theorem is presented in \eqref{sec:celldecomp}. Topological fields with a generic derivation are introduced in Section \ref{sec:dpmingen}: consistency results are presented in \eqref{sec:consistency}; relative quantifier elimination is given in \eqref{sec:relQE}.
Section \ref{sec:OP-EI} is devoted to the open core property (\ref{sec:opencore}), the $\delta$-cell decomposition theorem (\ref{sec:delta-cell-decomp}), and the transfer of elimination of imaginaries (\ref{sec:EI}). The applications to dense pairs are presented in Section \ref{sec:app}.  Some transfer proofs which were known in the one-sorted case (such as the transfer of $\NIP$ and distality) are gathered in Appendix \ref{appendix}, together with the elimination of the quantifier $\exists^{\infty}$. 

The reader eager of differential algebra may take the following fast track reading of Section \ref{sec:topopen}, before starting with Section \ref{sec:dpmingen}: read the definition of open $\cL$-theories of topological fields in Section \ref{sec:opentheories} and have a quick look at the cell decomposition theorem given for open $\cL$-theories of topological fields (Theorem \ref{thm:newCD}) in Section \ref{sec:celldecomp}.

\section{Open expansions of topological fields}\label{sec:topopen}  

\subsection{Preliminaries}\label{sec:prelim}

\subsubsection{Model theory} We follow standard model theoretic notation and terminology. Lower-case letters like $a, b, c$ and  $x,y,z$ usually denote finite tuples and we let $\ell(x)$ denote the length of $x$. We sometimes use $\bar{x}$ to denote a tuple $\bar{x}=(x_1,\ldots, x_n)$ where each $x_i$ is a tuple. Let $\cL$ be a possibly multi-sorted language and $\cM$ be an $\cL$-structure. For $A\subseteq \cM$ (possibly ranging over different sorts), we let $\langle A\rangle_\cL$ denote the $\cL$-substructure generated  by $A$ in $\cM$. Sorts are in general denoted by bold letters like $\bS$, and $\bS(\cM)$ denotes the $\bS$-points in $\cM$. For a tuple of variables $x$ (again possibly ranging over different sorts), we let $\cM^x$ denote the corresponding product of sorts in $\cM$. For an $\cL$-formula $\varphi(x)$ we let $\varphi(\cM)$ denote the set 
\[
\{a\in \cM^x : \cM\models \varphi(a)\}.
\]
We let $\cL(A)$ denote the extension of $\cL$ by constants for all elements in $A$. By an $\cL$-definable set of $\cM$ we mean definable with parameters, that is, of the form $\varphi(\cM)$ for an $\cL(\cM)$-formula $\varphi$. 

For a subset $X\subseteq Q\times R$ and for $a\in Q$, the fiber of $X$ over $a$ is denoted by $X_{a}\coloneqq\{b\in R : (a, b)\in X\}$. 

We let $\acl$ denote the model-theoretic algebraic closure operator. Given a sort $\bS$ in $\cL$, we let $\acl_\bS$ denote the model-theoretic algebraic closure restricted to $\bS$. Note that $\acl_\bS$ is again a closure operator. 

\subsubsection{Topological fields}\label{sec:topfields}

Throughout this article, every topological field is assumed to be non-discrete and Hausdorff. 

Let $K$ be a field and $\tau$ be a topology on $K$ making it into a topological field. On $K^n$, $n\geq 1$, we put the product topology. The topological closure of a set $X\subseteq K^n$ is denoted by $\overline{X}$ and its interior by $\Int(X)$. The topological dimension of a non-empty subset $X\subseteq K^n$, denoted $\Dim(X)$, is defined as the maximal $\ell\leqslant n$ such that there is a projection $\pi\colon K^n\rightarrow K^{\ell}$ such that $\Int(\pi(X))\neq \emptyset$ (and equal to $-1$ if $X=\emptyset$).

We let $\cL_{\mathrm{ring}}$ denote the language of rings $\{\cdot,+,-,0,1\}$ and $\cL_{\mathrm{field}}\coloneqq\cL_{\mathrm{ring}}\cup\{{\,}^{-1}\}$ denote the language of fields. We treat every field as an $\cL_{\mathrm{field}}$-structure by extending the multiplicative inverse to 0 by $0^{-1}=0$. We let $\cL_{\mathrm{field}}^\Omega$ be the extension of $\cL_{\mathrm{field}}$ by a set $\Omega$ of additional constant symbols (allowing the trivial case $\Omega=\emptyset$). 

For the rest of the article, unless otherwise stated, we work in a (possibly multi-sorted) language $\cL$ extending $\cL_{\mathrm{field}}^\Omega$. We distinguish the field sort and denote it by $\bF$. Any other sort is called an \emph{auxiliary sort}. 

\begin{notation}\label{nota:sorts} Let $\cK$ be an $\cL$-structure. Unconventionally, and in order to simplify notation, we use non-calligraphic $K$ to denote $\bF(\cK)$ (so $K$ is \emph{not} the underlying universe of $\cK$ but only the universe of the field sort of $\cK$).  
\end{notation}

\begin{definition}\label{def:deftop}
We say $\tau$ is an \emph{$\cL$-definable field topology} if there is an $\cL$-formula $\chi_\tau(x,z)$ with $x$ a single $\bF$-variable such that $\{ \chi_\tau(\cK,a): a\in \cK^z\}$ is a basis of neighbourhoods of 0.  
\end{definition}

If $K$ is an ordered field and the order is $\cL$-definable, then the order topology on $K$ is an $\cL$-definable field topology. Similarly, if $(K,v)$ is a valued field and the relation $\Div\coloneqq\{(x,y)\in K^2 : v(x)\leqslant v(y)\}$ is $\cL$-definable, then the valuation topology on $K$ is an $\cL$-definable field topology. When $\cK$ is a dp-minimal field, a result of W. Johnson \cite[Theorem 1.3]{johnson18} guarantees the existence of a definable $V$-topology on $K$ (equivalently a field topology induced either by a non-trivial valuation or an absolute value \cite{PZ}).

\subsection{Open theories of topological fields}\label{sec:opentheories}

We work in a first-order setting of topological fields which follows the spirit of \cite{pillay87}, \cite[Section 2]{vandendries1989}, \cite{M} and \cite{guzy-point2010}. The main new ingredient of the present account is that we explicitly allow multi-sorted structures.

\begin{definition}\label{def:T}
Fix an $\cL$-structure $\cK_0$ with $K_0$ a field of characteristic 0. Suppose that 
\begin{enumerate}
\item the restriction of $\cL$ to the sort $\bF$ is a relational extension of $\cL_{\mathrm{field}}^\Omega$, 
\item for every $\bF$-valued $\cL$-term $t(x,z)$ with $x$ a tuple of $\bF$-variables and $z$ a tuple of auxiliary sort variables, there is an $\bF$-valued $\cL$-term $\tilde{t}(x)$ such that 
\[
\cK_0\models (\forall z)(\forall x)(t(x,z)=\tilde{t}(x)),
\]
\item $K_0$ has an $\cL$-definable field topology.
\end{enumerate}
Let $T$ be the $\cL$-theory of $\cK_0$. Any such theory $T$ is called an \emph{$\cL$-theory of topological fields}. 
\end{definition}

\begin{remark}\label{rem:terms}
In a model $\cK$ of an $\cL$-theory of topological fields, it follows by (1) and (2) of the previous definition, that for any $A\subseteq \cK$  
\[
\langle A\rangle_\cL \cap K = \langle A\cap K \rangle_{\cL_\mathrm{field}^\Omega}. 
\]
\end{remark}

\begin{definition}
An $\cL$-theory of topological fields $T$ is called an \emph{open $\cL$-theory of topological fields} if it satisfies in addition the following condition: 
\begin{enumerate}[leftmargin=*, label={($\mathbf{A}$)}] 
\item\label{condA} for every 
$\cL$-formula $\varphi(x,z)$ with $x$ a tuple of $\bF$-variables and $z$ a tuple of auxiliary sort variables, there are $\cL$-formulas $\{\xi_h(z)\}_{h\in H}$ with $H$ a finite set, such that $\varphi(x,z)$ is equivalent modulo $T$ to
\[
\bigvee_{h\in H}\left(\xi_h(z)\rightarrow \left(\bigvee_{i\in I_h}\bigwedge_{j\in J_{ih}} P_{ijh}(x)=0 \wedge \theta_{ih}(x,z)\right)\right)
\] 
where $I_h$ and $J_{ih}$ are finite sets, $P_{ijh}\in \Q(\Omega)[x]\setminus \{0\}$ and $\theta_{ih}(x,z)$ is an $\cL$-formula such that for every model $\cK$ of $T$ and every $a\in \cK^z$,  $\theta_{ih}(\cK,a)$ defines an open set.\footnote{In \cite{guzy-point2010}, only the case when $\cL$ is one-sorted was considered and $K_{0}$ was called a \emph{topological $\cL$-field} in case the additional relation symbols and their complement were interpreted in $K_0$ as the union of an open set and a Zariski closed set.} 
\end{enumerate}
\end{definition}

\begin{remark}\label{rem:sorts} Suppose $T$ is an open $\cL$-theory of topological fields. 
\begin{itemize}[leftmargin=*]
\item Let $T'$ be an expansion by definitions of $T$ in a relational language $\cL'\supseteq \cL$. Then $T'$ is an open $\cL'$-theory of topological fields. In particular, the Morleyization $T_{\Mor}$ of $T$ is an open $\cL_{\Mor}$-theory of topological fields with quantifier elimination.  
\item Given a collection of sorts $\mathcal{G}$ of $\cL^\eq$, let $\cL^\mathcal{G}$ be the restriction of $\cL^\eq$ to the sorts in $\mathcal{G}$. Then, the $\cL^\mathcal{G}$-theory of some (any) model of $T$ is an open $\cL^\mathcal{G}$-theory of topological fields.   
\end{itemize}
\end{remark}

\begin{examples}\label{examples} The theory $T=\mathrm{Th}(\cK_0)$ is an open $\cL$-theory of topological fields in the following cases.
\begin{enumerate}[leftmargin=*] 
\item When $\cK_0$ is a real closed field and $\cL$ is $\cL_{\mathrm{of}}$ the language of ordered fields ${\cL_{\mathrm{field}}\cup\{<\}}$. The definable topology is given by the order topology. We use $\RCF$ for $T$. 

\item When $\cK_0$ is a henselian valued field of characteristic 0 and $\cL$ is the one-sorted language of valued fields given by $\cL_{\Div}=\cL_{\mathrm{field}}\cup\{\Div\}$. The definable topology corresponds to the valuation topology. Condition \ref{condA} follows by quantifier elimination in the multi-sorted $\cL_\RV$-language as defined in \cite{flenner}\footnote{Without loss of generality, we abuse of notation and use $\cL_\RV$ both for the residue characteristic 0 variant and the residue characteristic $p$ variant.}. For convenience, we will hereafter assume that $\cL_\RV$ contains $\cL_{\mathrm{div}}$ in the field-sort. Examples include algebraically, $p$-adically and real closed valued fields, and theories of classical valued fields such as $\mathbb{C}(\!(t)\!)$, $\mathbb{R}(\!(t)\!)$. 
\item Let $\cL_\RV'$-be an $\RV$-extension of $\cL_\RV$, i.e., it only extends the $\RV$-sorts. Any complete $\cL_\RV'$-expansion $T'$ of a complete theory $T$ of henselian valued fields of characteristic 0 is an open $\cL_\RV'$-theory of topological fields. Indeed, in such expansions we still have relative quantifier elimination of field quantifiers and, for a model $\cK$ of $T'$, since the preimage in $K$ of a non-zero point in an $\RV$-sort under its canonical quotient map is an open subset of $K$, condition \ref{condA} is also satisfied. 
\item The following open $\cL$-theories $T$ have, in addition, quantifier elimination in a one-sorted language: 
\begin{itemize}
    \item when $K_0$ is an algebraically closed valued field of characteristic $(0,p)$ with $p\geqslant 0$ and $\cL=\cL_\Div$ (we write $\ACVF_{0,p}$ for $T$); 
    \item when $K_0$ is a real closed valued field and $\cL=\cL_{\mathrm{ovf}}$ the language of ordered valued fields $\cL_{\mathrm{of}}\cup\{\Div\}$ (we write $\RCVF$ for $T$). \item when $K_0$ is a $p$-adically closed field of $p$-rank $d=ef$ and $\cL$ is 
\[
\cL_{p, d}\coloneqq\cL_{\mathrm{field}}\cup \{\Div, c_{1},\cdots,c_{d}\}\cup \{P_{n}: n\geqslant 1\}
\] 
as defined in \cite{pre-ro-84} (abusing of notation, we use in this case $\PCF$ for $T$).
\end{itemize} 
\end{enumerate}
\end{examples}

\begin{remark}\label{rem:notdpmin} 
Observe that not all theories in Examples \ref{examples} are dp-minimal. Indeed, there are various henselian fields of equicharacteristic 0 which are not dp-minimal. By a result of F. Delon \cite{delon81} combined with results of Y. Gurevich and P. H. Schmitt \cite{gurevich_schmitt}, the Hahn valued field $k(\!(t^\Gamma)\!)$ is $\NIP$ if and only if $k$ is $\NIP$ (as a pure field). Even assuming $\NIP$, by a result of A. Chernikov and P. Simon in \cite{chernikov_simon}, when $k$ is algebraically closed, the field $k(\!(t^\Gamma)\!)$ is dp-minimal if and only if $\Gamma$ is dp-minimal. However, there are ordered abelian groups which are not dp-minimal, as follows by a characterization of pure dp-minimal ordered abelian groups due to F. Jahnke, P. Simon and E. Walsberg in \cite[Proposition 5.1]{jahnke_simon_walsberg_2017}. 
\end{remark}

\begin{question} Is there an open $\cL$-theory of topological fields whose topology does not come from an order or a valuation (equivalently, is not a $V$-topology)? 
\end{question} 

\subsection{Basic consequences}\label{sec:basicconsec}
For the rest of Section \ref{sec:topopen}, unless otherwise stated, we let $T$ be an open $\cL$-theory of topological fields and $\cK$ be a model of $T$. By definable we mean $\cL$-definable. We gather some standard consequences that will be used throughout the article. Proofs are either elementary or follow by classical arguments, so they will be omitted. Let us first introduce some notation.  

\begin{notation}\label{nota:zariski}
Let $x=(x_0,\ldots, x_{n-1})$ be a tuple of $\bF$-variables and $y$ be a single $\bF$-variable. Let $\cA$ be a finite subset of $K[x,y]$ and $R\in K[x,y]$. We let the $\cL_{\mathrm{ring}}(K)$-formula $Z_\cA^R(x,y)$ be 
\begin{equation*}
Z_\cA(x,y)\wedge R(x,y)\neq 0. 
\end{equation*}
We write $Z_\cA(x,y)$ for $Z_\cA^1(x,y)$ and  given $P\in K[x,y]$, we let $Z_P(x,y)$ denote $Z_{\{P\}}(x,y)$. Finally, we define $\cA^y \coloneqq\{P\in \cA\mid \deg_y(P)>0\}$. 
\end{notation}

Note that every locally Zariski-closed subset of $K^{n+1}$ is of the form $Z_\cA^R(\cK)$ for some $\cA$ and some $R$. The following lemma follows by induction on degrees and iterated applications of Euclid's algorithm.

\begin{lemma}\label{cor:goodform} Every locally Zariski closed subset of $K^{n+1}$ can be written as a union of sets of the form $Z_{\cB}^{S}(\cK)$ where $S\in K[x,y]$ and either $\cB\subseteq K[x]$, or $\cB^y=\{P\}$ and $\frac{\partial}{\partial y} P$ divides $S$. \qed
\end{lemma}

From condition \ref{condA} we directly obtain the following corollary on the form of definable sets in the field sort. 

\begin{corollary}\label{cor:niceform1} Every definable subset of $K^{n+1}$ is defined by 
\[
\bigvee_{j\in J} Z_{\cA_j}^{S_j}(x,y) \wedge \theta_j(x,y)
\]
where $x=(x_0,\ldots, x_{n-1})$ are $\bF$-variables, $y$ a single $\bF$-variable, and for each $j\in J$, $\theta_j$ is an $\cL(\cK)$-formula that defines an open subset of $K^{n+1}$, $S_j\in K[x,y]$, and either 
\begin{enumerate}
\item $\cA_j\subseteq K[x]\setminus\{0\}$ or 
\item  $\cA_j\subseteq K[x,y]\setminus\{0\}$, $\cA_j^y=\{P_j\}$ and $\frac{\partial}{\partial y} P_j$ divides $S_j$. \qed
\end{enumerate}
\end{corollary}

The following proposition follows from condition \ref{condA} by standard arguments and classical results from \cite{vandendries1989}. 

\begin{proposition}\label{prop:conseAA1} The field sort of every model of $T$ is algebraically bounded (in the sense of \cite{vandendries1989}). The algebraic dimension $\mathrm{alg}\Dim$ (in the sense of \cite[Lemma 2.3]{vandendries1989}) on the field sort coincides with $\Dim_{\mathrm{acl}_\bF}$ and defines a dimension function (as defined in \cite{vandendries1989}). In particular, $T$ eliminates the field sort quantifier $\exists^{\infty}$ and the restriction $T'$ of $T$ to the field sort is a geometric theory. \qed
\end{proposition} 

The following lemma follows by a simple induction on $n$.

\begin{lemma}\label{lem:int_zar} Let $P(x)\in K[x]\setminus\{0\}$ with $x=(x_0,\ldots,x_{n-1})$. Then $\dim(Z_P(\cK))<n$.  In particular, the set $K^n\setminus Z_P(\cK)$ is open and dense in $K^n$ with respect to the ambient topology.\qed 
\end{lemma}

As a corollary of the previous results one can show the following consequence on the topological dimension. 

\begin{proposition}\label{prop:conseAA2} For every $n\geqslant 1$ and every definable subset $X\subseteq K^n$, $\Dim(X)=\Dim_{\mathrm{acl}_K}(X)$. The topological dimension satisfies thus the following properties for definable sets $X, Y\subseteq K^n$: 
\begin{enumerate}[(D1)]
\item $\Dim(X)=0$ if and only if $X$ is finite and non-empty,  
\item $\Dim(X\cup Y)=\max\{\Dim(X),\Dim(Y)\}$.
\item $\Dim(\overline{X}\setminus X)<\Dim(X)=\Dim(\overline{X})$,
\item $\Dim$ is additive, that is: for a non-empty definable set $X\subseteq  K^{m+n}$ and $d \in \{0,1, \ldots, n\}$, 
\[
\dim(\!\!\bigcup_{a \in X(d)} X_a)= \dim(X(d)) + d, 
\]
where $X(d) \coloneqq \{ a \in K^m : \dim X_a = d \}$. In addition, $X(d)$ is definable.\qed
\end{enumerate}
\end{proposition} 

As a convention, given an definable set $X$, we say that a property holds \emph{almost everywhere on $X$} if there is a definable subset $Y\subseteq X$ such that $\dim(X\setminus Y)<\dim(X)$ and the property holds on $Y$.

\subsection{Almost continuity of definable correspondences}\label{sec:correspon}
We will show in Section \ref{sec:celldecomp} a cell decomposition theorem for open $\cL$-theories of topological fields. An essential part of its proof is to show that definable correspondences (``multi-valued functions'') are almost everywhere continuous. We show such a result adapting to our setting the strategy employed by Simon-Walsberg for dp-minimal theories \cite{simon-walsberg2016}. Note that correspondences cannot be avoided in the absence of finite Skolem functions. Let us recall their definition.

\begin{definition}[{\cite[Section 3.1]{simon-walsberg2016}}] \label{def:correspondence} 
A {\it correspondence} $f\colon E\rightrightarrows K^\ell$ consists of a definable set $E\subseteq K^n$ together with a definable subset $\mathrm{graph}(f)$ of $E\times K^\ell$ such that
\[
0< \vert\{y\in K^\ell : (x,y)\in \mathrm{graph}(f)\}\vert<\infty, \text{ for all } x\in E.
\]
The set $\{y\in K^\ell : (x,y)\in \mathrm{graph}(f)\}$ is also denoted by $f(x)$. For a positive integer $m$, we say $f$ is an \emph{$m$-correspondence} if $|f(x)|=m$ for all $x\in E$. The correspondence $f$ is {\it continuous} at $x\in E$ if for every open set $V\subseteq K^\ell$ containing $f(x)$, there is an open neighbourhood $U$ of $x$ such that $f(U)\subseteq V$. 
\end{definition}
Let $f\colon E\subseteq K^n\rightrightarrows K^{\ell}$ be a correspondence. As a convention, when $n=0$, we identify $\mathrm{graph}(f)$ with a finite subset of $K^{\ell}$. If $E$ is an open subset and $\ell=0$, then we identify $\mathrm{graph}(f)$ with the set $E$. Note that a $1$-correspondence can be identified with a function. The following lemma is a reformulation of \cite[Lemma 3.1]{simon-walsberg2016}. 

\begin{lemma}\label{lem:corre_local_cont}
Let $U \subseteq K^n$ be open and let $f \colon U \rightrightarrows K^\ell$ be a continuous $m$-correspondence. Every $a\in U$ has an open neighbourhood $V$ such
that there are continuous functions $g_1,\ldots, g_m\colon V \to M^\ell$ such that 
$\Graph(g_i) \cap \Graph(g_j)=\emptyset$ when $i\neq j$ and
\[
\Graph(f\restriction V)=  \Graph(g_1) \cup \cdots \cup \Graph(g_m).  
\]
In addition, if $f$ is definable, we can further choose $V$ and the functions $g_i$ to be definable. 
\qed
\end{lemma}

The following result is a reformulation of \cite[Proposition 3.7]{simon-walsberg2016} in which we isolate the components of its proof in an axiomatic way. Recall that a family of sets $F$ is said to be \emph{directed} if for every $A,B\in F$ there is $C\in F$ such that $A\cup B\subseteq C$.  

\begin{proposition}[{\cite[Proposition 3.7]{simon-walsberg2016}}]\label{prop:continuity} Let $T$ be an $\cL$-theory of topological fields (not necessarily satisfying condition \ref{condA}) and $\cK$ be a model of $T$. Suppose $\cK$ satisfies the following properties 
\begin{enumerate}
\item if $A$ is a non-empty definable open subset of $K^n$ and $B$ is a definable subset of $A$ which is dense in $A$, then $\Int(B)$ is dense in $A$; in particular, $\dim(A\setminus B)<\dim(A)$. 
\item if $A, A_{1},\ldots, A_{k}$ are non-empty definable subsets of $K^n$,  $A$ is open and $A=\bigcup_{i=1}^{k} A_{i}$, then, $\Int(A_{i})\neq \emptyset$ for some $1\leqslant i\leqslant k$.   
\item if $C\subseteq K^{m+n}$ is a definable set such that the definable family $\{C_{a} : a\in K^m\}$ is a directed family and $\bigcup_{a\in K^m} C_{a}$ has non-empty interior, then, there is $a\in K^m$ such that $C_a$ has non-empty interior.
\item There is no infinite definable discrete subset of $K^{n}$.
\end{enumerate}
Then,  for $V \subseteq K^n$ a definable open set, every definable correspondence $f\colon V \rightrightarrows K^\ell$ is continuous on an open dense subset of $V$, and thus is continuous almost everywhere on $V$.
\end{proposition} 

\begin{proof} The proof is a word by word analogue of the proof of {\cite[Proposition 3.7]{simon-walsberg2016}} after replacing \cite[Lemma 2.6]{simon-walsberg2016} by condition (1), \cite[Corollary 2.7]{simon-walsberg2016} by condition (2), \cite[Lemma 3.5]{simon-walsberg2016} by condition (3) and \cite[Lemma 3.6]{simon-walsberg2016} by condition (4). 
\end{proof}

We now show that all four conditions in Proposition \ref{prop:continuity} hold for open $\cL$-theories of topological fields. 

\begin{lemma}\label{lem:condi1} If $A$ is a definable open subset of $K^n$ and $B$ is a definable subset of $A$ which is dense in $A$, then $\Int(B)$ is dense in $A$. In particular, $\dim(A\setminus B)<\dim(A)$. 
\end{lemma} 

\begin{proof}
It suffices to show that $\Int(B)\neq\emptyset$. By condition \ref{condA}, $B$ is defined by a formula of the form $\bigvee_{i\in I}\bigwedge_{j\in J_i} P_{ij}(x)=0 \wedge \theta_i(x)$, where $P_{ij}\in K[x]\setminus \{0\}$ and $\theta_i(x)$ defines an open set. Suppose that the set of indices $J_i$ is non-empty. Therefore, the algebraic dimension of $B$ is strictly less than $n$. Since the topological and the algebraic dimension coincide, $\dim(B)<n$. By Proposition \ref{prop:conseAA2}, $\dim(A\setminus B)=n$ which implies there is an open subset $U\subseteq A$ disjoint from $B$. This contradicts the density of $B$ in $A$. Then, there must be $i\in I$ such that $J_i=\emptyset$, hence $\Int(B)\neq\emptyset$.
\end{proof}

\begin{lemma}\label{lem:condi3} Suppose $C\subseteq K^{m+n}$ is a definable set inducing a definable family $\{C_{a} : a\in K^m\}$ which is directed. If $\bigcup_{a\in K^m} C_{a}$ has non-empty interior, then there is $a\in K^m$ such that $C_a$ has non-empty interior.
\end{lemma} 

\begin{proof}
Let $\varphi(x,y)$ with $\ell(x)=m$ and $\ell(y)=n$ be an $\cL(\cK)$-formula defining $C$. Let $Y\subseteq K^n$ denote the definable set $\bigcup_{a\in K^m} C_a$. By hypothesis, $\Int(Y)\neq\emptyset$. Since the family $\{C_a: a\in K^m\}$ is directed, we may assume there are infinitely many different $C_{a}$ in the family (as otherwise the result follows directly from Proposition \ref{prop:conseAA2}).

For $y=(y_1,\ldots, y_n)$, we let $\tilde y$ denote the tuple $(y_1,\ldots, y_{n-1})$. By Corollary \ref{cor:niceform1} applied to the formula $\varphi(x,y)$ with respect to the variable $y_{n}$,  $\varphi(x,y)$ is equivalent to a finite disjunction $\bigvee_{i\in I} \varphi_i$ where each $\varphi_i$ is of the form  $Z_{\cA_i}^{S_i}(x,y)\wedge \theta_i(x,y)$ where $\theta_i(K)$ defines an open subset of $K^{m+n}$ and either 
\begin{enumerate}
\item $\cA_i\subseteq K[x, \tilde y]$, or  
\item $\cA_i^{y_{n}}=\{P_i\}$ and $\frac{\partial }{\partial y_{n}} P_i$ divides $S_i$. 
\end{enumerate}
Collect all the subformulas of the disjunction of form (1) (resp. form (2)) and denote by $\varphi^{1}(x,y)$ (resp. $\varphi^{2}(x,y)$) their disjunction.
We have that $\varphi(x,y)=\varphi^1(x,y)\vee \varphi^2(x,y)$. 

Note that if $\cA_i=\emptyset$ for some $i\in I$, then each fiber $C_a$ contains the open set $\theta_i(a,\cK)$ and has therefore non-empty interior. Thus, we may assume that $\cA_i\neq \emptyset$ for all $i\in I$. We proceed by induction on $n$. Let $d$ be the maximum of the degrees (in $y_{n}$) of the polynomials occurring in all $\cA_i$'s. 

Assume $n=1$. Suppose first that $\cA_i\subseteq K[x]$ for some $i\in I$. Then, the fiber $C_a$ contains an open set whenever $Z_{\cA_i}^S(a,\cK)\neq \emptyset$. If $Z_{\cA_i}(a, \cK)= \emptyset$ for every $a\in K^m$, then we remove the corresponding member from the disjunction. Therefore, we are left with the case where $\varphi(x,y)=\varphi^2(x,y)$. We show this case cannot happen. First, note that in this situation each fiber $C_a$ has finite cardinality bounded by $d|I|$. Since the family $\{C_a:a\in K^m\}$ is directed, there is $a_{0}\in K^m$ such that $\varphi(a_{0}, \cK)=Y$. But this contradicts that $Y$ contains an open set (and is thus infinite). This concludes the case $n=1$.

Now assume $n>1$. Let $\pi\colon K^{n}\to K^{n-1}$ denote the projection onto the first $n-1$ coordinates. For $(a,u)\in K^{m}\times K^{n-1}$ we denote by $C_{a,u}$ the fiber $(C_a)_u=\{b\in K: (a,u,b)\in C\}$. By the form of each formula $\varphi_i$, each fiber $C_{a,u}$ either contains a non-empty open subset or is finite (and bounded by $d|I|$). We uniformly partition the projection $\pi(C_a)$ of each fiber $C_a$ into sets $\pi(C_a)_1$ and $\pi(C_a)_2$ where 
\begin{align*}
& \pi(C_a)_1\coloneqq\{u\in K^{n-1}: C_{a,u} \text{ contains an open set }\}\text{ and } \\
& \pi(C_a)_2\coloneqq\{u\in K^{n-1}: |C_{a,u}|\leqslant d|I|\}. 
\end{align*}
Since the definition is uniform, we have 
\[
\pi(Y)= \pi(\bigcup_{a\in K^m} C_a)= \bigcup_{a\in K^m} \pi(C_a) = \bigcup_{a\in K^m} \pi(C_{a})_1 \cup \bigcup_{a\in K^m} \pi(C_{a})_2. 
\]
As $Y$ contains an open set, so does $\pi(Y)$. Therefore, by Proposition \ref{prop:conseAA2}, either 
\[
\pi(Y)_2\coloneqq\bigcup_{a\in K^m} \pi(C_{a})_2\setminus (\bigcup_{a\in K^m} \pi(C_{a})_1) \text{ or } \pi(Y)_1\coloneqq\bigcup_{a\in K^m} \pi(C_{a})_1, 
\] 
contains an open set. 

\begin{claim}The set  $\pi(Y)_1$ must contain an open set. 
\end{claim}
\noindent Suppose for a contradiction $\dim(\pi(Y)_1)<n-1$. Partition $Y$ into $Y_1\cup Y_2$ where 
\[
Y_i=\{(u,b)\in K^n: u\in \pi(Y)_i\} \text{ for $i=1,2$ }.
\] 
By Proposition \ref{prop:conseAA2}, $Y_1$ or $Y_2$ contains an open subset. By construction, we have that $\pi(Y_1)=\pi(Y)_1$ and, by assumption, $\dim(\pi(Y)_1)<n-1$. Therefore $\dim(Y_1)<n$. Hence we must have that $\dim(Y_2)=n$. By the additivity of the dimension function (Proposition \ref{prop:conseAA2}), there must be $u\in \pi(Y_2)$ such that the fiber $(Y_2)_u$ is infinite. In particular, there are $k\in \N$ and $a_1,\ldots, a_k\in K^m$ such that $\bigcup_{j=1}^k C_{a_j,u}$ has cardinality bigger than $d|I|$. Since the family $\{C_{a,u}: a\in K^m\}$ is directed, there is $a\in K^m$ such that $\bigcup_{j=1}^k C_{a_j,u}\subseteq C_{a,u}$, which contradicts the fact that the fiber $C_{a,u}$ has cardinality smaller or equal than $d\vert I\vert$. This completes the claim. 

\medskip

Consider the directed family $\{\pi(C_{a})_1: a\in K^m\}$. By the claim, $\pi(Y)_1=\bigcup_{a\in K^m} \pi(C_{a})_1$ contains a non-empty open set. Therefore, by induction, there is $a\in K^m$ such that $\pi(C_{a})_1$ contains an open subset, say $U\subseteq K^{n-1}$. For each $i\in I$, set 
\[
U_i\coloneqq\{u\in U : \dim(Z_{\mathcal{A}_i}^{S_i}(a,u,\cK)\cap\theta_i(a,u,\cK))=1 \}. 
\]
Given that $U\subseteq \pi(Y)_1$, we have that $U=\bigcup_{i\in I} U_i$. Then, by Proposition \ref{prop:conseAA2}, there is $i\in I$ such that $U_i$ contains an open set, say $V\subseteq U_i$. By the definition of $U_i$, the set 
$Z_{\mathcal{A}_i}^{S_i}(a,u,\cK)\cap \theta_i(a,u,\cK)$ is infinite for every $u\in V$. Thus, since $\mathcal{A}_i\neq \emptyset$, we must have $\mathcal{A}_i\subseteq K[x,\tilde y]$. But in this situation $Z_{\mathcal{A}_i}(a,\cK)=K^n$. Indeed, consider $\mathcal{B}=\mathcal{A}_i$ as a set of polynomials in $K[x,\tilde y]$. Then $Z_\mathcal{B}(a,\cK)\subseteq K^{n-1}$ contains an open set, namely, $V$. By Lemma \ref{lem:int_zar}, $Z_{\mathcal{B}}(a,\cK)=K^{n-1}$, and thus $Z_{\mathcal{A}_i}(a,\cK)=K^n$ (as a subset of $K^n$). Therefore, the fiber $C_a$ contains the open set 
\[
\{(u,b)\in V\times K: S_i(a,u,b)\neq 0 \wedge \theta_i(a,u,b)\}.
\]
This is indeed open, as the set defined by the formula $S_i(x,\tilde y,y_n)\neq 0\wedge \theta_i(x,\tilde y,y_n)$ defines a non-empty open subset of $K^{m+n}$ by Lemma \ref{lem:int_zar}. 
\end{proof}

\begin{lemma}\label{lem:condi4} There is no infinite definable discrete subset of $K^{n}$.
\end{lemma} 

\begin{proof} The proof goes by induction on $n$ exactly as the proof of \cite[Lemma 3.6]{simon-walsberg2016}. The base case follows directly straightforward by condition \ref{condA}. The inductive case follows word by word the proof in \cite{simon-walsberg2016}. 
\end{proof}

\begin{proposition}\label{prop:acont} For a definable open set $V \subseteq K^n$, every definable correspondence $f\colon V \rightrightarrows K^\ell$ is continuous almost everywhere.
\end{proposition} 

\begin{proof} This follows from Proposition \ref{prop:continuity} where conditions (1)-(4) correspond respectively to Lemma \ref{lem:condi1}, Proposition \ref{prop:conseAA2}, Lemma \ref{lem:condi3} and Lemma \ref{lem:condi4}. 
\end{proof}

\subsection{Cell decomposition}\label{sec:celldecomp}

We finish this section with a cell decomposition theorem for open $\cL$-theories of topological fields. As in Simon-Walsberg's \cite[Proposition 4.1]{simon-walsberg2016}, cells correspond to definable continuous $m$-correspondences (for some $m$). However, the main difference with their result is that we provide an inductive definition of cells which remains closer to classical definitions (e.g., o-minimal cells). In particular, our definition of cell ensures that if $X\subseteq K^n$ is a cell and $\pi\colon K^n\to K^d$ is a coordinate projection onto the first $d$-variables, then $\pi(X)$ is also a cell. 

In order to define cells, we use the following conventions. Let $X\subseteq K^n$,  $\rho\colon K^n\to K^d$ be a coordinate projection and $f\colon \rho(X) \rightrightarrows K^{n-d}$ be a correspondence. Let $\rho^\perp\colon K^n \to K^{n-d}$ be the complement projection of $\rho$ (i.e. to the complement set of coordinates). The \emph{graph of $f$ along $\rho$} is the set
\[
\{x\in K^n : \rho(x)\in \rho(X) \text{ and } \rho^\perp(x)\in f(\rho(x))\}.
\]

\begin{definition}[Cells]\label{def:cells} For $A\subseteq \cK$, we define the collection of $A$-definable cells of $K^n$ by induction on $n$ as pairs $(X,\rho_X)$ where $X\subseteq K^n$ is a non-empty $A$-definable set and $\rho_X\colon K^n \to K^{\dim(X)}$ is a coordinate projection such that $\rho_X(X)$ is an open set (with the convention that if $\dim(X)=0$, then $\rho_X(X)$ is open).  
\begin{enumerate}[leftmargin=*]
\item A pair $(X,\rho_X)$ with $X\subseteq K$ is an $A$-definable cell if and only if $X$ is an $A$-definable non-empty open set and $\rho_X=\id$, or $X$ is an $A$-definable non-empty finite set and $\rho_X$ is the projection to $K^0$. 

\item Assume $A$-definable cells have been defined for all $k\leqslant n$ and let $\pi\colon K^{n+1}\to K^n$ be the projection onto the first $n$-coordinates. A pair $(X,\rho_X)$ with $X\subseteq K^{n+1}$ is an $A$-definable cell if and only if $X$ is the graph along $\rho_X$ of a continuous $A$-definable $m$-correspondence $f_X\colon \rho_X(X)\rightrightarrows K^{n+1-\dim(X)}$ (for some $m>0$), there is an $A$-definable cell $(C, \rho_C)$ such that $C=\pi(X)$ and one of the following three cases holds
\begin{enumerate}[(i)]
\item $\rho_X=\id$ (so $X$ is an $A$-definable open set); 
\item $\rho_X=\rho_C\circ\pi$; 
\item $\rho_X$ corresponds to 
\[
\rho_X(x_1,\ldots, x_{n+1}) = (\rho_C(x_1,\ldots, x_n), x_{n+1}),  
\]
the fiber $X_c$ is a non-empty open subset of $K$ for every $c\in C$, and for every $x\in X$
\[
f_X(\rho_X(x)) = f_C(\rho_C(\pi(x))). 
\]
\end{enumerate}
\end{enumerate}
By a cell, we mean an $\cK$-definable cell. We often omit the associated projection $\rho_X$ of a cell, and simply write $X$ for a cell. 
\end{definition}

\begin{lemma}\label{lem:subcell} Let $(X,\rho_X)$ be an $A$-definable cell with $X\subseteq K^n$. If $Y\subseteq \rho_X(X)$ is a non-empty open $A$-definable subset, then the set $X'=\rho_X^{-1}(Y)\cap X$ together with the associated projection $\rho_{X'}=\rho_{X}$, is an $A$-definable cell. 
\end{lemma}

\begin{proof} 
We proceed by induction on $n$. If $X\subseteq K$, the result is obvious. So suppose the results holds for all $k\leqslant n$ and let $X\subseteq K^{n+1}$. We proceed by induction on $\dim(X)$. If $X$ is finite there is nothing to show. So suppose the result for all $i<\dim(X)$. Since $X$ is a cell, it is the graph along $\rho_X$ of a continuous $m$-correspondence $f_X\colon \rho_X(X)\rightrightarrows K^{n+1-\dim(X)}$. Let $\pi\colon K^{n+1}\to K^n$ be the projection onto the first $n$ coordinates and $(C,\rho_C)$ be the corresponding cell such that $C=\pi(X)$. If $X$ is of type (i) ($\dim(X)=n+1$), then $X$ is open and the result is again obvious. Hence we may suppose that $\dim(X)<n+1$. 

Suppose first $X$ is of type (ii) as given in Definition \ref{def:cells}. Then $\rho_X=\rho_C\circ\pi$ and $f_X\colon \rho_C(C)\rightrightarrows K^{n+1-\dim(X)}$. In particular, $Y\subseteq \rho_X(X) = \rho_C(C)$, so by induction $C'=\rho_C^{-1}(Y)\cap C$ is a cell with $\rho_{C'}=\rho_{C}$. So $X'$ is the graph of $f_{X}\restriction \rho_{C'}(C')$ along $\rho_{X'}$, hence a cell of type (ii). 

It remains to show the result when $X$ is of type (iii) as given in Definition \ref{def:cells}. In this case, $\rho_X(x_1,\ldots, x_{n+1})=(\rho_C(x_1,\ldots, x_n), x_{n+1})$, for every $c\in C$, the fiber $X_c$ is open and $f_X(\rho_X(x))=f_C(\rho_C(\pi(x)))$ for all $x\in X$. Let $\pi'\colon \rho_X(X)\to K^{\dim(C)}$ be the coordinate projection such that $\rho_C \circ \pi = \pi'\circ \rho_X$ (i.e., dropping the last coordinate). Since $\pi'(Y)\subseteq \rho_C(C)$ is open, by induction, $C'=\rho_C^{-1}(\pi'(Y))\cap C$ is a cell with $\rho_{C'}=\rho_C$. Note that $\pi(X') = C'$ and it holds that $X_c'=X_c$ for every $c\in C'$ (so in particular $X_c'$ is open). Since $X'$ is the graph of $f_X\restriction \rho_{X'}(X')$ and $\rho_{X'}(X)=Y$, $X'$ is a cell of type (iii). 
\end{proof}

\begin{theorem}[Cell decomposition]\label{thm:newCD} Let $T$ be an open $\cL$-theory of topological fields, $\cK$ be a model of $T$ and $A\subseteq \cK$. Let $X$ be an $A$-definable subset of $K^{n+1}$, then $X$ can be expressed as a finite disjoint union of $A$-definable cells.  
\end{theorem}

\begin{proof}
Let $X\subseteq K^{n+1}$ be an $A$-definable set. We proceed by induction on $n$. 

For $n=0$, note that $X=\Int(X)\cup (X\setminus \Int(X))$. Both sets are $A$-definable, and since $\dim(X\setminus \Int(X))<\dim(X)$,  this already constitutes a cell decomposition. 

Suppose now the result for all $k\leqslant n$ and proceed by induction on $\dim(X)$. Let $\pi\colon K^{n+1}\to K^n$ be the projection onto the first $n$-coordinates. 

 If $\dim(X)=0$, then $X$ is already a cell, so suppose the result holds for all $i<\dim(X)$. If $\dim(X)=n+1$, then after removing $\Int(X)$ from $X$ (which is an open definable subset of $K^{n+1}$ and so a cell), we may already suppose $\dim(X)<n+1$. Writing an element in $X$ as a pair $(a,b)$ with $a\in \pi(X)$ and $b\in K$, we partition $X$ into the following two $A$-definable subsets of $X$: 
\begin{equation}\label{eq:x1x2}
W = \bigcup_{a\in H} \{a\} \times \Int(X_a) \text{ and } Z = X\setminus W,  
\end{equation}
where $H = \{a\in \pi(X) : \dim(X_a)=1 \}$. Since $W$ and $Z$ form an $A$-definable partition of $X$, it suffices to show the result for $W$ and $Z$. 

\medskip

Assume that $\dim(X)=\dim(Z)$ (if $\dim(Z)<\dim(X)$ the result,  follows by induction). By the definition of $Z$ and the fact that for every $a\in H$, $\dim(X_a\setminus \Int(X_a))=0$, the fiber $Z_a$ is finite for every $a\in \pi(Z)$. Hence, by elimination of $\exists^\infty$, there is a uniform bound on the size of $Z_a$. Possibly further partitioning $Z$, we may suppose that every fiber $Z_a$ is of size $m$ for some $m>0$. By induction, since $\pi(Z)\subseteq K^{n}$, we may express $\pi(Z)$ as a finite disjoint union of cells $(C,\rho_C)$. Let us consider each cell in turn and assume that $\pi(Z)=C$. Since $C$ is the graph along $\rho_C$ of a continuous correspondence $f_C \colon \rho_C(C) \rightrightarrows  K^{n-\dim(C)}$ (note $C$ cannot be finite by assumption on $\dim(Z)$), $Z$ is the graph along $\rho_C\circ\pi$ of an $m'$-correspondence $f \colon \rho_C(C) \rightrightarrows  K^{n+1-\dim(C)}$ (for some $m'>0$). By Proposition \ref{prop:acont}, $f$ is continuous on an open dense definable subset $U$ of $\rho_C(C)$. Then we partition $Z$ as a disjoint union of 
\[
Z'\coloneqq \mathrm{graph}(f\restriction U) \text{ and } Z''\coloneqq \mathrm{graph} (f\restriction (\rho_C(C)\setminus U)).
\]  
Setting $\rho_{Z'}=\rho_C\circ\pi$, since $f$ is continuous on $U$, $Z'$ is the graph along $\rho_{Z'}$ of an $m'$-continuous correspondence $f_{Z'}=f\restriction U$ and therefore a cell of type (ii) by Lemma \ref{lem:subcell}. For $Z''$ the result follows by induction since $\dim(Z'')<\dim(Z)$.

\medskip

It remains to show the result for $W$. Assume that $\dim(X)=\dim(W)$ (if $\dim(W)<\dim(X)$, the result follows by induction). By the definition of $W$, $W_a=\Int(X_a)$ is a non-empty open subset of $K$, for each $a\in \pi(W)=H$. By induction, we may suppose that $\pi(W)$ is a finite disjoint union of cells $(C,\rho_{C})$. 
As in the previous case we consider each cell in turn and further assume that $\pi(W)=C$. Since $C$ is a cell, $C$ is the graph along $\rho_C$ of a continuous $m$-correspondence $f_C\colon \rho_C(C)\rightrightarrows K^{n-\dim(C)}$, for some $m$. Note that $C$ cannot be an open set since $\dim(X)<n+1$. Set $\rho_W\colon K^{n+1}\to K^{\dim(W)}$ to be the coordinate projection given by 
\[
\rho_{W}(x_1,\ldots,x_{n+1})\coloneqq (\rho_C(x_1,\ldots,x_n),x_{n+1}), 
\]
and $\pi'\colon \rho_W(W)\to K^{\dim(C)}$ be the coordinate projection such that $\rho_C \circ \pi = \pi'\circ \rho_W$. In particular, $W$ is the graph along $\rho_W$ of the $m$-correspondence 
\[
f\colon \rho_W(W)\rightrightarrows K^{n+1-\dim(X)}\colon x \mapsto f_C(\pi'(x)). 
\]
Set $Y'=\Int(\rho_W(W))$. Note that $f\restriction Y'$ is continuous. Set $W'\subseteq W$ to be the graph of $f\restriction Y'$ along $\rho_W$ and $W''\coloneqq W\setminus W'$. The set $W'$ is already a cell of type (iii) with $\rho_{W'}=\rho_W$. Indeed, by Lemma \ref{lem:subcell}, $C'=\rho_C^{-1}(\pi'(Y'))\cap C$ is a cell with $\rho_{C'}=\rho_{C}$ and $\pi(W')=C'$; for every $a\in C'$ we have that $W_a=W_a'$, so the fiber is open; and by construction $W'$ is the graph along $\rho_{W'}$ of a continuous $m$-correspondence $f_{W'}=f\restriction Y'$ satisfying $f_{W'}(\rho_{W'}(x))=f_C(\rho_C(\pi(x)))$ for all $x\in W'$. For $W''$ the result follows by induction since 
\[
\dim(W'')=\dim(\rho_W(W)\setminus \Int(\rho_W(W)))<\dim(\rho_W(W))=\dim(W).\qedhere
\]
\end{proof}

The reader may have noticed that Theorem \ref{thm:newCD} is not really about open $\cL$-theories of topological fields. Indeed, such a result holds in any structure with a uniform definable topology such that (i) the topological dimension satisfies the properties of Proposition \ref{prop:conseAA2} and (ii) definable correspondences are almost everywhere continuous. In particular, it applies to all dp-minimal structures with a uniform definable topology as in \cite{simon-walsberg2016} (i.e. satisfying properties (1) and (2) in Corollary \ref{cor:DC-dp-min} below) which satisfy the exchange property. Indeed, (i) follows from Proposition \cite[Proposition 2.2]{simon-walsberg2016} and (ii) from \cite[Proposition 3.7]{simon-walsberg2016}, together with the fact that the exchange property implies that the dimension is additive \cite[Exercise 4.38]{simon-nip}. 

\begin{corollary}\label{cor:DC-dp-min} Let $\cM$ be a dp-minimal structure with a uniform definable topology on a distinguished home sort $M$ satisfying 
\begin{enumerate}
    \item $M$ does not have any isolated points;
    \item every infinite definable subset of $M$ has non-empty interior;
    \item $\acl_{M}$ has the exchange property. 
\end{enumerate}
Then, every definable subset of $M$ is a finite disjoint union of definable cells (as defined in Definition \ref{def:cells}).  \qed
\end{corollary}

We finished this section with a technical lemma that will be later needed in Section \ref{sec:opencore}. 

\begin{lemma}\label{lem:Ldens}
Let $X\subseteq K^n$ be a cell. Let $1\leqslant d<n$, $\pi\colon K^n\to K^{n-d}$ be the projection onto the first $n-d$ coordinates and $a=(a',a'')\in X$ with $a'=\pi(a)$. Then, given an open neighbourhood $V$ of $a''$, there exists an open neighbourhood $U$ of $a'$ such that for all $b'\in U\cap \pi(X)$, there is $b''\in V$ such that $b=(b',b'')\in X$.
\end{lemma}
\begin{proof}
Let us first show the proposition for $d=1$. Let $C\subseteq K^{n-1}$ be the cell such that $\pi(X)=C$ and $f_X\colon \rho_X(X)\rightrightarrows K^{n-\dim(X)}$ be a continuous $m$-correspondence such that $X$ is the graph of $f_X$ along $\rho_X$. We split in cases depending on the form of $X$. 

\begin{itemize}[leftmargin=*]
\item If $X$ is open (cell of type (i)), the result is straightforward.  

\item Assume that $X$ is a cell of type (ii), so $\rho_X=\rho_C \circ\pi$. Let $\pi^{\perp}\colon K^{n-1-\dim(C)} \to K$ be the projection onto the last coordinate. By Lemma \ref{lem:corre_local_cont}, there is an open neighbourhood $W$ of $\rho_X(a)=\rho_C(a')$ and continuous definable functions $g_1,\ldots,g_m\colon W \to K^{n-\dim(X)}$ such that 
$\Graph(g_i) \cap \Graph(g_j)=\emptyset$ when $i\neq j$ and
\[
\Graph(f_X\restriction W)=  \Graph(g_1) \cup \cdots \cup \Graph(g_m).  
\]
Without loss of generality, suppose $g_1$ is such that $a''=\pi^\perp(g_1(\rho_X(a)))$. Let $D$ be an open neighbourhood of $g_1(\rho_X(a))$ such that $\pi^\perp(D)\subseteq V$. Since $g_1$ is continuous, we can find an open neighbourhood $O$ of $\rho_X(a)$ such that $g_1(O)\subseteq D$. Set $U\coloneqq \rho_C^{-1}(O)$ and let $b'\in U\cap \pi(X)=U\cap C$. Then, $\rho_C(b')\in O$, and hence $g_1(\rho_C(b'))\in D$. By the choice of $D$, $b''=\pi^\perp(g_1(\rho_C(b')))\in V$, which shows that $b=(b',b'')\in X$, by the definition of $g_1$. 

\item Assume now that $X$ is a cell of type (iii), so that $\rho_X$ corresponds to 
\[
\rho_X(x_1,\ldots, x_n) = (\rho_C(x_1,\ldots, x_{n-1}), x_n), 
\] 
the fiber $X_c$ is a non-empty open subset of $K$ for every $c\in C$ and, $f_X(\rho_X(x))=f_C(\rho_C(\pi(x)))$ for all $x\in X$. Since $\rho_X(a)=(\rho_C(a'),a'')$ and $\rho_X(X)$ is open, we can find open neighbourhoods $O$ of $\rho_C(a')$ and $V'$ of $a''$ such that $O\times V'\subseteq \rho_X(X)$ and $V'\subseteq V$. Set $U=\rho_C^{-1}(O)$ and take $b'\in U\cap \pi(X)=U\cap C$. Let $b''$ be any element in $V'$. Then, $\rho_X((b',b''))=(\rho_C(b'),b'')\in O\times V'\subseteq \rho_X(X)$, and by the definition of $X$, since $b'\in C$,  
\[
f_C(\rho_C(b')) = f_X(\rho_X(b',b'')),  
\]
this shows that $(b',b'')\in X$. 
\end{itemize}

Now we consider the general case and we proceed by induction on $d$. So assume the property has been proven for all $n$ and all $1\leqslant e<d$. Let $\pi_{n-1}\colon K^n\to K^{n-1}$ be the projection onto the first $n-1$ coordinates and recall that $\pi$ denotes the projection onto the first $n-d$ coordinates. Write $a=(a_1,\ldots, a_n)$ as a triple $(a_1', a_2', a_n)$ where 
\[
a_1'=\pi(a) \text{ and } a_2'=(a_{n-d+1},\ldots, a_{n-1}).
\]
Without loss of generality suppose $V=V_1\times V'$ with $V'$ an open neighbourhood of $a_n$ and $V_1$ an open neighbourhood of $a_2'$. By the case $d=1$, there is an open neighbourhood $U'$ of $\pi_{n-1}(a)$ such that for any $b'\in U\cap \pi_{n-1}(X)$ there is $b''\in V'$ such that $(b',b'')\in X$. By induction on $\pi_{n-1}(X)$, for any open neighbourhood $V_2$ of $a_2'$ there is an open neighbourhood $U_1$ of $a_1'$ such that for all $b_1\in U_1\cap \pi(X)$, there is $b_2\in V_2$ such that $(b_1,b_2)\in \pi_{n-1}(X)$.
Possibly shrinking $U_{1}$ and $V_{2}$, we may suppose $U_1\times V_2\subseteq U'$. Let us show that $U=U_1$ has the desired property. Indeed, if $b_1\in U\cap \pi(X)$, then there is $b_2\in V_2$ such that $(b_1,b_2)\in \pi_{n-1}(X)$. Setting $b'=(b_1,b_2)$, since $b'\in U_1\times V_2\subseteq U'$, there is $b''\in V'$ such that $(b',b'') = (b_1,b_2, b'')\in X$. By construction, $(b_2,b'')\in V_2\times V'=V$, which completes the proof.  
\end{proof}

\section{Theories of topological fields with a generic derivation}\label{sec:dpmingen}

Let $T$ be an $\cL$-theory of topological fields and $\cL_\delta$ be the language $\cL$ extended by a symbol for a derivation (in the field sort). From now on we will focus on the study of an $\cL_\delta$-extension $T_\delta^*$ of $T$. The derivation $\delta$ of any model of $T_\delta^*$ is called a \emph{generic derivation}. Every such a derivation is highly non-continuous. The theory $T_\delta^*$ will be defined in Section \ref{sec:Tdelta}. In Section \ref{sec:consistency} we show various examples of theories $T$ for which the theory $T_\delta^*$ is consistent. Then, in Section \ref{sec:relQE} we show that if $T$ is an open $\cL$-theory of topological fields and $T$ eliminates $\bF$-quantifiers, $T_\delta^*$ also eliminates $\bF$-quantifiers.

Before defining $T_\delta^*$, let us fix some notation and recall the needed background on differential algebra. We work in a model $\cK$ of $T$. 

\subsection{Differential algebra background}
Equip $K$ with a derivation $\delta$, that is, an additive morphism $\delta\colon K\to K$ which satisfies Leibniz rule $\delta(ab)=\delta(a)b+a\delta(b)$. We let $C_K$ denote the field of constants of $K$, namely, $C_K\coloneqq\{a\in K :\delta(a)=0\}$. It is a subfield of $K$. 

For~$m\geqslant 0$ and~$a\in K$, we define 
\[
\delta^m(a)\coloneqq\underbrace{\delta\circ\cdots\circ\delta}_{m \text{ times}}(a), \text{ with $\delta^0(a)\coloneqq a$,}
\]
and~$\bar{\delta}^m(a)$ as the finite sequence~$(\delta^0(a),\delta(a),\ldots,\delta^m(a))\in K^{m+1}$. Similarly, given an element~$a=(a_1,\ldots,a_n)\in K^n$ and a tuple of non-negative integers $\bar{m}=(m_1,\ldots, m_n)$, we write 
\begin{align*}
\delta^m(a) & \coloneqq (\delta^m(a_1), \ldots, \delta^m(a_m)) \\ 
\bd^{m}(a)& \coloneqq(\bd^{m}(a_1),\ldots, \bd^{m}(a_n)) \\     
\bd^{\bar{m}}(a)& \coloneqq (\bd^{m_1}(a_1),\ldots, \bd^{m_n}(a_n)).     
\end{align*}
For notational clarity, we sometimes use $\J_{\bar{m}}$ instead of $\bd^{\bar{m}}$, especially concerning the image of subsets of $K^n$. We abuse of notation and write $K^{\bar{m}+1}$ for the product $K^{m_1+1}\times\cdots\times K^{m_n+1}$. With this notation, for $A\subseteq K^n$, we have that $\J_{\bar{m}}(A)=\{\bar{\delta}^{\bar{m}}(a): a\in A\}\subseteq K^{\bar{m}+1}$.

We assume that our tuples of variables are ordered with the convention that, given variables $x=(x_0,\ldots,x_n)$ and a variable $y$, the tuple $(x,y)$ is ordered such that $y$ is bigger than $x_n$ (and similarly when $y$ is an ordered tuple of variables).

For $x$ as above, we let $K\{x\}$ be the ring of differential polynomials in $n+1$ differential indeterminates $x_{0},\ldots, x_{n}$ over $K$, namely it is the ordinary polynomial ring in formal indeterminates $\delta^j(x_{i})$, $0\leqslant i\leqslant n$, $j\in \N$, with the convention $\delta^0(x_{i})\coloneqq x_{i}$. We extend the derivation $\delta$ to $K\{x\}$ by setting $\delta(\delta^i(x_j))=\delta^{i+1}(x_j)$. By a rational differential function we simply mean a quotient of differential polynomials.

\subsubsection{Order and separant of a differential polynomial}\label{sec:order_sep}
For $P(x)\in K\{x\}$ and $0\leqslant i\leqslant n$, we let $\ord_{x_i}(P)$ denote the \emph{order of $P$ with respect to the variable $x_i$}, that is, the maximal integer $k$ such that $\delta^k(x_i)$ occurs in a non-trivial monomial of $P$ and $-1$ if no such $k$ exists. We let \emph{the order of $P$} be the tuple  
\[
\ord(P)\coloneqq (\ord_{x_0}(P), \ldots, \ord_{x_n}(P)) 
\]
Similarly, for a finite subset $\cA$ of $K\{x\}$, we let 
\[
\ord_{x_i}(\cA)\coloneqq \max\{\ord_{x_i}(P) : P \in \cA\}, \text{ and } \ord(\cA)\coloneqq(\ord_{x_1}(\cA), \ldots, \ord_{x_n}(\cA)). 
\]
For $R\in K\{x\}$, we write $\ord_{x_i}(\cA,R)$ for $\ord_{x_i}(\cA\cup\{R\})$. For $P$ as above, let $m=(m_1,\ldots, m_n)$ be the tuple given by $m_i=\max\{\ord_{x_i}(P), 0\}$, 
that is, we replace all occurrences of $-1$ in $\ord(P)$ by zeros. Let $\bar{x}=(\bar{x}_0,\ldots,\bar{x}_n)$ be a tuple of variables such that $\ell(x_i)=m_i+1$. We let $P^*\in K[\bar{x}]$ denote the corresponding ordinary polynomial such that $P(x)=P^*(\bd^{\bar{m}}(x))$. 

Suppose $\ord_{x_n}(P)=m\geqslant 0$. Then, there are (unique) differential polynomials $c_i\in K\{x\}$ such that $\ord_{x_n}(c_i)<m$ and 
\begin{equation}\label{eq:diffpol}
P(x)=\sum_{i=0}^d c_i(x)(\delta^m(x_n))^i. 
\end{equation}
The separant $s_{P}$ of $P$ is defined as $s_{P}\coloneqq\frac{\partial}{\partial \delta^m(x_{n})}P\in K\{x\}$. We extend the notion of separant to arbitrary polynomials with an ordering on their variables in the natural way, namely, if $P\in K[x]$, the separant of $P$ corresponds to $s_P\coloneqq\frac{\partial}{\partial x_{n}}P\in K[x]$. By convention, we induce a total order on the variables $\delta^j(x_{i})$ by declaring that 
\[
\delta^k(x_{i})< \delta^{k'}(x_{j}) \Leftrightarrow 
\begin{cases}
i<j \\
i=j \text{ and } k<k'. 
\end{cases}  
\]
This order makes the notion of separant for differential polynomials compatible with the extended version for ordinary polynomials, \emph{i.e.}, 
$s_{P^*} = s_P^*$.

\subsubsection{Minimal differential polynomials}\label{sec:gen_poly}

Let $F\subseteq K$ be an extension of differential fields and $x$ be a single variable. Recall that an ideal $I$ of $F\{x\}$ is a \emph{differential ideal} if for every $P\in I$, $\delta(P)\in I$. For $a\in K$, let $I(a,F)$ denote the set of differential polynomials in $F\{x\}$ vanishing on $a$. The set $I(a,F)$ is a prime differential ideal of $F\{x\}$. Let $P\in I(a,F)$ be a differential polynomial of minimal degree among the elements of $I(a,F)$ having minimal order. Any such differential polynomial is called a \emph{minimal differential polynomial of $a$ over $F$}. Let $\langle P\rangle$ denote the differential ideal generated by $P$ and $I(P)\coloneqq\{Q(x)\in F\{ x\}: s_{P}^{\ell}Q\in \langle P\rangle$ for some $\ell\in \N\}$. 

\begin{lemma}[{\cite[Section 1]{marker1996}}]\label{lem:generic} If $P$ is a minimal differential polynomial for $a$ over $F$, then $I(a,F)=I(P)$. \qed  
\end{lemma}

\subsubsection{Rational prolongations}\label{sec:rat_prolong}
For $x=(x_0,\ldots, x_n)$ we define an operation sending $P\mapsto P^\delta$ for $P\in K\{x\}$ such that $\ord_{x_n}(P)\geqslant 0$ as follows: for $P$ written as in \eqref{eq:diffpol}
\[
P(x)\mapsto P^\delta(x) = \sum_{i=0}^d \delta(c_i(x)) (\delta^m(x_n))^i. 
\] 
A simple calculation shows that 
\begin{equation}\label{eq:diffpol2}
\delta(P(x)) = P^\delta(x) + s_P(x)\delta^{m+1}(x_n).
\end{equation}

\begin{lemma-definition}\label{lemdef:ratioprolong} Let $x=(x_0,\ldots,x_n)$ be a tuple of variables and $y$ be a single variable. Let $P\in K\{x,y\}$ be a differential polynomial such that $m=\ord_y(P)\geqslant 0$. There is a sequence of rational differential functions $(f_i^P)_{i\geqslant 1}$ such that for every $a\in K^{n+1}$ and $b\in K$
\[
K\models [P(a,b)=0 \wedge s_P(a,b)\neq 0] \to \delta^{m+i}(b)=f_i^P(a,b).
\]
In addition, each $f_i^P$ is of the form 
\[
f_i^P(x,y) = \frac{Q_i(x,y)}{s_P(x,y)^{\ell_i}},
\]
where ${\ell_i}\in \N$, $\ord_{y}(Q_i)=\ord_y(P)$ and 
\[
\ord_{x_j}(Q_i)=
\begin{cases}
\ord_{x_j}(P)+i & \text{ if $\ord_{x_j}(P)\geqslant 0$ }\\
-1 & \text{ otherwise }
\end{cases}
\]
We call the sequence $(f_i^P)_{i\geqslant 1}$ the \emph{rational prolongation along $P$}. 
\end{lemma-definition} 

\begin{proof} It suffices to inductively define the polynomials $Q_i$. By \eqref{eq:diffpol2}, if $\delta(P(x,y))=0$ we obtain that
\[
\delta^{m+1}(y) = \frac{-P^{\delta}(x,y)}{s_P(x,y)}, 
\]
Setting $Q_1=-P^{\delta}$, the rational differential function $f_1^P=\frac{Q_1}{s_P}$ satisfies the required property. Now suppose $Q_{i}$ has been defined and that $f_{i}^P=\frac{Q_i}{(s_P)^{\ell_i}}$ satisfies $\delta^{m+i}(y)=f_i^P(x,y)$. By applying $\delta$ on both sides we obtain
\begin{align*}\label{eq:prolongation}
\delta^{m+i+1}(y) & = \frac{\delta(Q_i(x,y))s_P(x,y) -  Q_i(x,y)\delta(s_P(x,y))}{s_P(x,y)^{2\ell_i}}.  
\end{align*}
By replacing instances of $\delta^{m+i}(y)$ in $\delta(Q_i(x,y))$ and $\delta(s_P(x,y))$ by $f_i^P(x,y)$, we obtain in the numerator a differential polynomial of order $m$ with respect to $y$. Setting $Q_{i+1}$ as such numerator shows the result. The last assertion is a straightforward calculation. 
\end{proof}

\begin{notation}\label{not:prolong} Let $x=(x_0,\ldots,x_m)$ be a tuple of variables. For an integer $d\geqslant 0$, we define a new tuple of variables $x(d)$ which extends $x$ by $d$ new variables, that is, 
\[
x(d)\coloneqq (x_0, \ldots, x_m, x_{m+1},\ldots, x_{m+d}).
\]  
When $\bar{x}=(\bar{x}_0,\ldots,\bar{x}_\ell)$ is a tuple of tuples of variables, we let $\bar{x}[d]\coloneqq (\bar{x}_0(d),\ldots,\bar{x}_\ell(d))$. Note that if $d\neq 0$ and $\bar{x}$ is not a singleton, then $\bar{x}[d]$ and $\bar{x}(d)$ are different.  
\end{notation} 

\begin{notation}\label{not:prolong2} Let $x=(x_0,\ldots, x_n)$ and $y$ be a single variable. Let $P\in K\{x,y\}$ be a differential polynomial with $m\coloneqq \ord_y(P)\geqslant 0$ and let $(f_i^P)_{i\geqslant 1}$ be its rational prolongation along $P$. Let $\bar{x}=(\bar{x}_0,\ldots,\bar{x}_n)$ where $\bar{x}_i=(x_{i},x_{i,1}\ldots,x_{i, m_i})$ with $m_i=\max\{\ord_{x_i}(P), 0\}$ and $\bar{y}=(y, y_1,\ldots,y_m)$. For every $d\geqslant 0$, we let $\lambda_{P}^d (\bar{x}[d], \bar{y}(d))$ be the $\cL(K)$-formula: 
\[
P^*(\bar{x}, \bar{y})=0 \wedge s_{P}^*(\bar{x}, \bar{y})\neq 0 \wedge \bigwedge_{i= 1}^d y_{m+i}=(f_i^P)^*(\bar{x}[i], \bar{y}).
\]
\end{notation} 

\subsubsection{Kolchin closed sets}\label{sec:kolchin}

Similarly as in Notation \ref{nota:zariski}, we introduce the following notation:
\begin{notation}\label{nota:kolchin}
Let $x$ be a tuple of variables with $\ell(x)=n$. for a finite subset $\cA$ of $K\{x\}$ and $R\in K\{x\}$, we let $\cZ_\cA^R(x)$ denote the $\cL_\delta(K)$-formula 
\[
\bigwedge_{P\in \cA} P(x)=0 \wedge R(x)\neq 0.
\]
We let $\cZ_\cA(x)$ be $\cZ_\cA^1(x)$. 
\end{notation}
Recall that a subset $X\subseteq K^n$ is called \emph{Kolchin closed} if there is a finite subset $\cA\subseteq K\{x\}$ such that $X=\cZ_\cA(\cK)$. It is called \emph{locally Kolchin closed} if $X=\cZ_{\cA}^R(\cK)$ for some $R\in K\{x\}$. The following lemma is the differential analogue of Lemma \ref{cor:goodform}. Its proof is completely algebraic and follows by induction on orders and degrees using Lemma \ref{cor:goodform}. We leave it to the reader. 

\begin{lemma}\label{lem:minord3} Let $x=(x_1,\ldots,x_n)$, $y$ be a single variable, $\cA$ be a finite subset of $K\{x,y\}$ and $R\in K\{x,y\}$. The set $\cZ_\cA^R(K)$ is the union of finitely many sets $\cZ_\cB^{S_\cB}(K)$, with $\cB$ a finite subset of $K\{x,y\}$ and $S_\cB\in K\{x,y\}$, such that $\ord_y(S_\cB)\leqslant \ord_y(\cA,R)$, $\ord_{x_i}(\cB,S_\cB)\leqslant \ord_{x_i}(\cA,R)+\ord_y(\cA,R)$ for each $1\leqslant i\leqslant n$, and either 
\begin{enumerate}
\item[$\bullet$] $\cB\subseteq K\{x\}$ or
\item[$\bullet$] there is a unique $P_{\cB}\in \cB$ of non-negative order in $y$, $\ord_y(P_\cB)\leqslant \ord_y(\cA,R)$ and $s_{P_\cB}$ divides $S_\cB$. \qed
\end{enumerate}
\end{lemma}

\subsection{The theory $T_\delta^*$}\label{sec:Tdelta}

Denote by $T_\delta$ the $\cL_{\delta}$-theory $T$ together with the usual axiom of a derivation, namely,  
\begin{equation*}
\forall x\forall y(\delta(x+y)=\delta(x)+\delta(y) \wedge \delta(xy)=\delta(x)y+x\delta(y))
\end{equation*}
and an axiom $\delta(c) = 0$ for each constant symbol $c\in \Omega$. Note that this is only one possible choice of how to define the derivation on $\Q(\Omega)$. For simplicity of presentation we focus on the case where $\Q(\Omega)$ is a subfield of the field of constants $C_K$ for every model $K$ of $T_\delta$. This choice will only play an important role in Section \ref{sec:app}. 

\begin{lemma-definition}\label{lem:etiole} Let $x=(x_1,\ldots,x_n)$ be a tuple of $\bF$-variables and $z$ be a tuple of auxiliary sort variables. Let $t(x,z)$ be an $\cL_\delta$-term and $\varphi(x,z)$ be an $\bF$-quantifier free $\cL_\delta$-formula. Then,  
\begin{enumerate}
    \item there are a tuple $\bar m= (m_1,\ldots, m_{n})$, $m_i\in \N$, a tuple of $\bF$-variables $\bar{x}=(\bar{x}_1,\ldots, \bar{x}_n)$ with $\ell(\bar{x}_i)=m_i+1$ and an $\cL$-term $t^*(\bar{x},z)$ such that 
    \[
    T_\delta\models (\forall x)(\forall z)(t(x,z)=t^*(\bd^{\bar m}(x),z)), \text{ and } \]
\item tuples $\bar m$ (depending on $\varphi$) and $\bar{x}$ as above together with an $\bF$-quantifier free $\cL$-formula $\varphi^*(\bar{x},z)$ such that 
\[
T_\delta \models (\forall x)(\forall z) (\varphi(x,z) \leftrightarrow \varphi^*(\bd^{\bar m}(x), z)). 
\]
\end{enumerate}
\end{lemma-definition}

\begin{proof}
Point (1) follows by induction on terms. The main step consists in showing the result for an $\cL_\delta$-term of the form $\delta(t_0(x,z))$ where $t_0(x,z)$ is an $\bF$-valued $\cL_\delta$-term. By induction, $t_0(x,z)=t_0^*(\bd^{\bar{m}}(x),z)$ modulo $T_\delta$ for some $\cL$-term $t_0^*(\bar{x},z)$. Then, by part (2) of Definition \ref{def:T}, there is an $\cL$-term $\widetilde{{t}_0^*}(\bar{x})$ such that $t_0^*(\bar{x},z)=\widetilde{{t}_0^*}(\bar{x})$ modulo $T$. Thus, it suffices to show the result for $\delta(\widetilde{{t}_0^*}(\bar{x}))$. This follows from Leibniz's rule and additivity of a derivation together with the fact that the restriction of $\cL$ to $\bF$ is a relational extension of $\cL_{\mathrm{field}}^\Omega$ (part (1) of Definition \ref{def:T}).

Part (2) follows by induction on formulas using part (1). 
\end{proof}
\par Note that when $t(x)$ is a differential polynomial $P(x)\in K\{x\}$, then the term $t^*$ coincides with the ordinary polynomial $P^*$, described in subsection \ref{sec:order_sep}.

\



We now describe a scheme of $\cL_{\delta}$-axioms generalizing the axiomatization of $\CODF$ given by M. Singer in \cite{singer1978}. Let $\chi_\tau(x,z)$ be the $\cL$-formula providing a basis of neighbourhoods of 0. Abusing of notation, when $x$ is a tuple of $\bF$-variables $x=(x_1,\ldots, x_n)$ we let $\chi_\tau(x,z)$ denote the formula 
\[
\bigwedge_{i=1}^n \chi_\tau(x_i,z).
\]

\begin{definition}\label{def:theoryT_delta} The $\cL_\delta$-theory $T_{\delta}^*$ is the union of $T_{\delta}$ and the following scheme of axioms $(\mathrm{DL})$: given a model $\cK$ of $T_\delta$, $\cK$ satisfies $(\mathrm{DL})$ if for every differential polynomial $P(x)\in K\{x\}$ with $\ell(x)=1$ and $\ord_x(P)=m\geqslant 1$, for field sort variables $y=(y_0,\ldots,y_m)$ it holds in $\cK$ that
\[
(\forall z)\big((\exists y)(P^*(y)=0 \land s_P^*(y)\ne 0 )
\rightarrow  \exists x\big(P(x)=0\land s_P(x)\ne 0\wedge
\chi_\tau(\bd^m(x)-y, z)\big)\big).
\]
\end{definition}

As usual, by quantifying over coefficients, the axiom scheme $(\mathrm{DL})$ can be expressed in the language $\cL_\delta$.  The scheme (DL) essentially expresses the following: given a differential polynomial $P(x)$ of order $m$, if its associated ordinary polynomial $P^*$ has a regular zero $a$, then for every open neighbourhood $U$ of $a$ we may find a zero $b$ of $P$ such that $\bd^m(b)$ is in $U$. 

\medskip

Letting $T=\RCF$, we get back $\CODF$ as the theory $\RCF_\delta^*$.

\subsection{Consistency}\label{sec:consistency} 

The main result of this section is Theorem \ref{thm:consistency} which shows that if $T$ is a complete theory of topological fields for which the topology is induced by a henselian valuation, then $T_\delta^*$ is consistent. As a consequence we obtain the consistency of $T_\delta^*$ for all theories $T$ described in Examples \ref{examples}. For some of such theories, the consistency of $T_\delta^*$ has already been proved. Indeed, the consistency of $\CODF$ was proved in \cite{singer1978} and was later generalized in \cite{Tressl} for large fields and then in \cite{guzy-point2010} for topological fields, using a notion of topological largeness. Although we will follow a very similar strategy to the known proofs, our argument is based on henselizations rather than using explicitly a notion of largeness (or topological largeness) for the fields under consideration. However, the fields for which consistency will be proven, all being henselian, are large fields. 

Let us start by a general criterion to show that $T_\delta^*$ is consistent. 

\begin{proposition}\label{prop:criterion} Let $T$ be a complete $\cL$-theory of topological fields and $\chi_\tau(x,z)$ be the $\cL$-formula defining a basis of neighbourhoods of 0. Suppose that for every model $\cK$ of $T$ and every derivation $\delta$ on $K$ the following holds 
\begin{enumerate}
\item[($\ast$)] for every $P\in K\{x\}$ ($\ell(x)=1$) of order $m\geqslant 1$ for which there is $a\in K^{m+1}$ such that $P^*(a)=0$ and $s_P^*(a)\neq 0$, there is an $\cL$-elementary extension $\cF$ of $\cK$, a derivation on $F$ extending $\delta$ and $b\in F$ such that $P(b)=0$, $s_P(b)\neq 0$ and for every $c\in \cK^z$
\[
\cF\models \chi_\tau(\bd^m(b) - a, c). 
\]
\end{enumerate} 
Then, for every model $\cK$ of $T$ and every derivation $\delta$ on $K$, there is an $\cL$-elementary extension $\cM$ of $\cK$ and an extension of $\delta$ to $M$ making $(\cM,\delta)$ into a model of $T_\delta^*$. In particular, $T_\delta^*$ is consistent. 
\end{proposition} 

\begin{proof} Fix some model $\cK$ of $T$ and some derivation $\delta$ on $K$. We use the following two step construction to build $(\cM,\delta)$.    

\emph{Step 1:} We construct an $\cL$-elementary extension $\cK\preccurlyeq_\cL \cF_\cK$ and a derivation on $F_\cK$ extending $\delta$ as follows. Let $(P_i)_{i< \lambda}$ be an enumeration of all differential polynomials $P_i\in K\{x\}$ with $\ord(P_i)=m_i\geqslant 1$ together with tuples
$a_i\in K^{m_i+1}$ such that $P_i^*(a_i)=0$ and $s_{P_i}^*(a_i)\neq 0$. Consider the following chain $(\cF_i, \delta_i)_{i<\lambda}$ defined by
\begin{enumerate}
\item[(i)] $(\cF_0,\delta_0)\coloneqq (\cK,\delta)$,
\item[(ii)] $(\cF_{i+1}, \delta_{i+1})$ is given by condition $(\ast)$ with respect to $(\cF_i,\delta_i)$ and $P_i$, that is, $\cF_i\preccurlyeq_\cL \cF_{i+1}$, $\delta_{i+1}$ extends $\delta_i$ and there is $b_i\in \cF_{i+1}$ such that $P(b_i)=0$, $s_P(b_i)\neq 0$ and for every $c\in \cK^z$
\[
\cF_{i+1}\models \chi_\tau(\bd^m(b_i) - a_i, c). 
\]
\item[(iii)]At limit stages, we take unions.
\end{enumerate}
Abusing notation, let $(\cF_\cK,\delta)\coloneqq \bigcup_{i<\lambda} \cF_i$, where $\delta$ denotes the union of the derivations $\delta_i$. Observe that indeed $\cK\preccurlyeq_\cL \cF_K$ and $(K,\delta)\subseteq (F_\cK,\delta)$ is an extension of differential fields.   

\emph{Step 2:} Define a chain $(\cM_i,\delta_i)_{i<\omega}$ where $(\cM_0,\delta_0)\coloneqq (\cK,\delta)$, and $(\cM_{i+1},\delta_{i+1})$ corresponds to $(\cF_{\cM_i},\delta_{i+1})$ obtained in Step (1) with respect to $(\cM_i,\delta_i)$, so that $\cM_i\preccurlyeq_\cL \cM_{i+1}$. Let $(\cM,\delta)\coloneqq \bigcup_{i<\omega} (\cM_i,\delta_i)$. By construction, $\cK\preccurlyeq_\cL \cM$. It is an easy exercise to show that $(\cM,\delta)$ satisfies the axiom scheme (DL).
\end{proof}

\begin{proposition}\label{prop:consistency} Let $\cL$ be an extension of $\cL_{\Div}$ and $T$ be a complete $\cL$-theory of topological fields where $\chi_\tau(x,z)$ corresponds to the $\cL_\Div$-formula $v(x)>v(z)\wedge z\neq 0$. Suppose that for every model $\cK$ of $T$, $(K,v)$ is henselian. Let $\delta$ be a derivation on $K$. Then, there is an $\cL$-elementary extension $\cM$ of $\cK$ and an extension of $\delta$ to $M$ making $(\cM,\delta)$ into a model of $T_\delta^*$. In particular, $T_\delta^*$ is consistent. 
\end{proposition} 

\begin{proof} It suffices to show condition $(\ast)$ in the statement of Proposition \ref{prop:criterion}. Let $\cK$ be a model of $T$ such that $K$ is equipped with a derivation $\delta$. Let $v$ denote the henselian valuation on $K$ and $\Gamma^v$ denote the value group of $(K,v)$. Suppose $P\in K\{x\}$ is a differential polynomial of order $m\geqslant 1$ and $a=(a_0,\ldots, a_m)\in K^{m+1}$ such that $P^*(a)=0$ and $s_P^*(a)\neq 0$. Let $\cK^*$ be an $\cL$-elementary extension of $\cK$ containing a tuple of elements $t=(t_0,\ldots, t_{m})$ such that $\Gamma^v<v(t_0)\ll v(t_1)\ll\cdots\ll v(t_m)$. Consider the (ordinary) polynomial 
\[
Q(x)=P^*(a_0-t_0, \ldots, a_{m-1}-t_{m-1}, x)
\] 
in $K(t_0,\ldots t_{m-1})[x]$. Let $w\colon K(t)\to \Z^n_\infty$ denote the coarsening of $v$ such that $w(K)=0$ and $0<w(t_0)\ll w(t_1)\ll\cdots\ll w(t_m)$. Let $F=(K(t), w)^h$ and $L=(K(t), v)^h$ be their corresponding henselizations. Without loss, we may suppose that $F\subseteq L\subseteq K^*$ as sets (the first inclusion might be assumed since $w$ is a coarsening of $v$) and that for all $y\in F$
\[
w(y)>0 \Leftrightarrow v(y)>\Gamma^v. 
\]
Let us show that there is $c\in F$ such that $Q(c)=0$ and $w(c-a_m)>0$. The reduction $\widetilde{Q}$ of $Q$ modulo the maximal ideal of $(F,w)$ corresponds to $P^*(a_0,\ldots,a_{m-1},x)\in K[x]$. By assumption, $\widetilde{Q}(a_m)=0$ and $\frac{\partial}{\partial x} \widetilde{Q}(a_m)\neq 0$. Then, by Hensel's lemma, there is $c\in F$ such that $Q(c)=0$ and $w(c-a_m)>0$ (equivalently $v(c-a_m)>\Gamma^v$). This implies both that $c\notin K$ and that $\frac{\partial}{\partial x}Q(c)\neq 0$. We extend $\delta$ to the subfield $K(t_0,\ldots, t_{m-1}, c)\subseteq F$ by inductively setting 
\begin{enumerate}
\item $\delta(t_i)=\delta(a_i)+t_{i+1}-a_{i+1}$ \text{ for $0\leqslant i < m-1$ },
\item $\delta(t_{m-1}) = c$,  
\end{enumerate}
noting that, since $Q(c)=0$ and $\frac{\partial}{\partial x} Q(c)\neq 0$, the derivative of $c$ is already determined by the rational prolongation $f_1^Q(c)$. Setting $b\coloneqq a_0-t_0$, we have that 
\[
P(b)=P^*(\bd^m(b))=P^*(a_0-t_0, \ldots, a_{m-1}-t_{m-1},c)=Q(c) = 0.
\]
Similarly, $s_P(b)\neq 0$. In addition, for every $e\in K^\times$, $K^*\models \chi_\tau(\bd^m(b) -a, e)$, since  
\begin{equation}\label{eq:valuationbig}
v(\bd^m(b) -a) = \min\{ v(t_0), \ldots, v(t_{m-1}), v(c-a_m)\} > v(e).  
\end{equation}
Extending the derivation from $K(t_0,\ldots,t_{m-1},b)$ to $K^*$ (such an extension always exists by \cite[Theorem 5.1]{lang}) completes the result.
\end{proof}

The following consistency result slightly improves Proposition \ref{prop:consistency} by allowing cases in which the definable topology is induced by a henselian valuation, but such a valuation is not necessarily definable. 

\begin{theorem}\label{thm:consistency} Let $T$ be a complete $\cL$-theory of topological fields. Assume every model $\cK$ of $T$ has the following property:  letting $\tau$ be the definable topology, there is an $\cL$-elementary extension $\cF$ of $\cK$ with $F$ equipped with a henselian valuation $v$ such that the valuation topology coincides with $\tau$. Then, for every derivation $\delta$ on $K$, there is an $\cL$-elementary extension $\cM$ of $\cK$ and an extension of $\delta$ to $M$ making $(\cM,\delta)$ into a model of $T_\delta^*$. In particular, $T_\delta^*$ is consistent. 
\end{theorem}

\begin{proof} Consider the language $\cL'=\cL\cup\{\Div\}$ and let $T'=\Th_{\cL'}(\cF)$. In particular, $T'$ is an $\cL'$-theory of topological fields with henselian valued fields in the field sort. Extend the derivation $\delta$ to some (any) a derivation on $F$. By Proposition \ref{prop:consistency}, there is an $\cL'$-elementary extension $\cM$ of $\cF$ and an extension of $\delta$ to $M$ such that $(\cM,\delta)\models (T')_\delta^*$. Note that $\cK\preccurlyeq_\cL \cF \preccurlyeq_\cL \cM$. Since the valuation topology coincides with $\tau$ on $M$, the reduct of $\cM$ to $\cL_\delta$ is a model of $T_\delta^*$. 
\end{proof}

\begin{corollary}\label{cor:cons_examples} Let $T$ be any theory from Examples \ref{examples}. Then $T_\delta^*$ is consistent. 
\end{corollary} 

\begin{proof} Except for $\CODF$, all examples in Examples \ref{examples} correspond to expansions of theories of henselian valued fields of characteristic 0 and the result follows already from Proposition \ref{prop:consistency}. For $\CODF$, the consistency follows by Theorem \ref{thm:consistency}, since any real closed valued field can be elementarily embedded into a real closed valued field (which is henselian) for which the order topology and the valuation topology coincide. Alternatively (and essentially the same), the consistency of $\RCVF_\delta^*$ implies the consistency of $\CODF$, as every model of $\RCVF_\delta^*$ is a model of $\CODF$. 
\end{proof}

\begin{remark}\label{rem:closedness} Let $(K,v)$ be a valued field of characteristic 0 endowed with a derivation $\delta$. Let $(K^h,v)$ be the henselization of $(K,v)$. Note that the derivation extends (uniquely) to $K^h$. Let $T$ be the theory of $(K^h,v)$. Proposition \ref{prop:consistency} implies that $(K,v,\delta)$ embeds as an $\cL_{\Div,\delta}$-structure into a model of $T_\delta^*$. 
\end{remark} 

\begin{remark}\label{rem:consistency} Note that if $T_\delta^*$ is consistent, then every model $\cK$ of $T$ embeds as an $\cL$-structure into a model of $T_\delta^*$. Indeed, take a model $\cK'$ of $T_\delta^*$. Then the reduct of $\cK'$ to $\cL$ is a model of $T$, and since $T$ is complete $\cK\equiv_\cL \cK'$. The result follows by Keisler-Shelah's theorem.
\end{remark}

\subsection{Relative quantifier elimination}\label{sec:relQE}

For the rest of Section \ref{sec:dpmingen} we let $T$ be an open $\cL$-theory of topological fields and assume $T_\delta^*$ is a consistent theory. 

We will need the following classical consequence of the axiom scheme (DL). 

\begin{lemma}[{\cite[Lemma 3.17]{guzy-point2010}}] \label{fact:density} Let $\cK$ be a model of $T_{\delta}^{*}$. Let $O$ be an open subset of $K^n$. Then there is $a\in K$ such that $\bar \delta^{n-1}(a)\in O$.\qed 
\end{lemma}

\begin{theorem}\label{thm:QE} Suppose $T$ eliminates field sort quantifiers. Then, $T_{\delta}^{*}$ eliminates field sort quantifiers.
\end{theorem}

\begin{proof} Let $\Sigma$ denote the set of $\bF$-quantifier free $\cL_\delta$-formulas. Let $x$ be a tuple of $\bF$-variables, $z$ a tuple of auxiliary sort variables, $y$ a single variable and $\varphi(x,y,z)$ be a formula in $\Sigma$. Let $\cK_1, \cK_2$ be two models of $T_\delta^{*}$, $b_i\in \cK_i^{x}$ and $c_i\in \cK_i^{z}$ such that 
\[
(b_1, c_1) \equiv_{\Sigma} (b_2, c_2).
\]
Suppose there is $a\in \cK_1^y$ such that $\cK_1\models\varphi(b_1,a,c_1)$. It suffices to show there is $a'\in \cK_2^y$ such that $\cK_1\models\varphi(b_2,a',c_2)$. If $y$ is an auxiliary sort variable, the result follows trivially, so suppose $y$ is an $\bF$-variable. 

\begin{claim}\label{cla:terms}
$\langle b_i, c_i \rangle_{\cL_\delta} \cap K_i = \langle (\delta^{n}(b_i))_{n\in\N} \rangle_{\cL_{\mathrm{field}}^\Omega}$. 
\end{claim}

By Part (1) of Lemma-Definition \ref{lem:etiole}, we have that $\langle b_i, c_i \rangle_{\cL_\delta} = \langle (\delta^n(b_i))_{n\in\N}, c_i \rangle_{\cL}$.  By Remark \ref{rem:terms}, $\langle (\delta^n(b_i))_{n\in\N}, c_i \rangle_{\cL}\cap K_i= \langle (\delta^n(b_i))_{n\in\N} \rangle_{\cL_{\mathrm{field}}^\Omega}$, which completes the claim.  

\medskip

Let $\cF_i$ denote $\langle (\delta^n(b_i))_{n\in\N} \rangle_{\cL_{\mathrm{field}}^\Omega}$ and $\sigma\colon \cF_1\to \cF_2$ be an $\cL_\delta$-isomorphism. By elimination of $\bF$-quantifiers in $\cL$ modulo $T$, we may extend $\sigma$, as an $\cL$-isomorphism, to $\cF_1\cup (F_1^\alg\cap K_1)$. In fact, such an extension is an $\cL_\delta$-isomorphism. Indeed, it follows from Lemma-Definition \ref{lemdef:ratioprolong}, that if $a_0\in F_1^\alg\cap K_1$ and $P$ is its minimal polynomial over $F_1$, then for each $n\in \N^*$, $\delta^n(a_0)=f^P_n(a_0)$ where $f_n^P$ is a rational function with coefficients in $F_1$.  

By part (2) of Lemma-Definition \ref{lem:etiole} let  $\varphi^*(\bar{x},\bar{y},z)$ be the $\bF$-sort quantifier free $\cL$-formula such that 
\[
T_\delta^*\models (\forall x)(\forall y)(\forall z)(\varphi(x,y,z) \leftrightarrow \varphi^*(\bd^{\bar{m}}(x),\bd^m(y),z) 
\]
for some tuple $\bar{m}$ and $m\geqslant 0$. By condition \ref{condA}, the formula $\varphi^*(\bar{x},\bar{y},z)$ is equivalent to
\[
\bigvee_{h\in H}\left(\xi_h(z)\rightarrow \left(\bigvee_{i\in I_h}\bigwedge_{j\in J_{ih}} P_{ijh}(\bar{x},\bar{y})=0 \wedge \theta_{ih}(\bar{x},\bar{y},z)\right)\right)
\] 
Let $h\in H$ such that $\xi_h(c_1)$ holds and $i\in I_h$ such that   
\begin{equation}\label{eq:intersection1}
\bigwedge_{j\in J_{ih}} P_{ijh}(\bd^{\bar{m}}(b_1), \bd^m(a))=0 \wedge \theta_{ih}(\bd^{\bar{m}}(b_1),\bd^m(a),c_1),
\end{equation}
holds. 


We split in two cases. 

\emph{Case 1:} Suppose $\ord_y(P_{ijh}(\bd^{\bar{m}}(b_1), \bd^m(y)))\geqslant 0$ for some $j\in J_{ih}$, so that $a$ is differentially algebraic over $F_1$. Let $P\in F_1\{y\}$ be a minimal differential polynomial for $a$ over $F_1$ of order $k$. We may assume $a$ is not algebraic over $F_1$, since $\sigma$ was already extended to $F_1^\alg\cap K_1$, and hence that $k\geqslant 1$. Since $P$ is minimal, we must have both $k\leqslant m$ and that $s_P(a)\neq 0$. For $d=m-k$, we have then (see Notation \ref{not:prolong2})
\begin{equation}\label{eq:minimality1}
\cK_1\models \lambda_P^{d}(\bd^{m}(a)). 
\end{equation}
Let $P^\sigma$ denote the corresponding polynomial over $F_2$ in which every coefficient of $P$ is replaced by its image under $\sigma$. Since $I(a, F_1)=I(P)$ (by Lemma \ref{lem:generic}), it holds for every $j\in J_{ih}$ such that $\ord_y(P_{ijh}(\bd^{\bar{m}}(b_1), \bd^m(y))) \geqslant 0$ that
\[
P_{ijh}(\bd^{\bar{m}}(b_1),\bd^m(y))\in I(P).
\] 
Our assumption on $b_1$ and $b_2$ implies that $P_{ijh}(\bd^{\bar{m}}(b_2),\bd^m(y))\in I(P^\sigma)$. Therefore, it suffices to show that there is $a'\in K_2$ such that 
\[
\cK_2 \models P^\sigma(a')=0 \wedge s_{P^\sigma}(a')\neq 0 \wedge \theta_{ih}(\bd^{\bar{m}}(b_2), \bd^m(a'),c_2),
\]
as this will also imply that $\cK_2\models \bigwedge_{j\in J_{ih}} P_{ijh}(\bd^{\bar{m}}(b_2), \bd^m(a'))=0$. By \eqref{eq:minimality1}, $\cK_1\models (\exists \bar{y})\lambda_{P}^d(\bar{y})$ (which is an $\cL(F_1)$-formula), so by $\bF$-quantifier elimination in $T$, there is $\bar{e}=(e_0,\ldots,e_m)\in K_2^{m+1}$ such that $\cK_2\models\lambda_{P^\sigma}^{d}(\bar{e})$. Letting $\widehat{e}=(e_0,\ldots,e_k)$, the previous formula yields that $e_{k+i}=(f_i^{P^\sigma})^*(\widehat{e})$ for all $1\leqslant i\leqslant d$, where $(f_i^{P^\sigma})_{i\geqslant 1}$ is the rational prolongation of $P^\sigma$. Since $(P^\sigma)^*(\widehat{e})=0$ and $s_{P^\sigma}^*(\widehat{e})\neq 0$, the axiom scheme (DL) implies there is $a'\in K_2$ such that  $P^\sigma(a')=0$ and $s_{P^\sigma}(a')\neq 0$. Moreover, by the continuity of the functions $(f_i^{P^\sigma})^*$, we may further suppose that $\theta_{ih}(\bd^{\bar{m}}(b_2), \bd^m(a'), c_2)$ holds. This completes Case 1.

\emph{Case 2:} Suppose $\ord_y(P_{ijh}(\bd^{\bar{m}}(b_1), \bd^m(y)))= -1$ for all $j\in J_{ih}$. Since the set $\theta_{ih}(\bd^{\bar{m}}(b_2),\cK_2,c_2)$ is an open subset of $K_2^{m+1}$, the result follows directly from Lemma \ref{fact:density}. 
\end{proof}

\begin{corollary}\label{cor:QE-theories} For $T$ either $\RCF$, $\ACVF_{0,p}$, $\RCVF$ or $\PCF$, $T_\delta^*$ has quantifier elimination. For $T$ the $\cL_\RV$-theory of a henselian valued field of characteristic 0, the theory $T_\delta^*$ eliminates $\bF$-quantifiers.  \qed
\end{corollary}

By Remark \ref{rem:sorts}, the theory $T_\Mor$ is an open $\cL_\Mor$-theory of topological fields, and therefore, by Theorem \ref{thm:QE}, the theory $(T_\Mor)_\delta^*$ eliminates $\bF$-quantifiers. We exploit this fact in what follows. 

\begin{corollary}\label{cor:complete} The theory $T_\delta^*$ is complete. 
\end{corollary}
\begin{proof} Note that $T_\delta^*$ is complete if and only if $(T_\Mor)_\delta^*$ is complete. Therefore, by Theorem \ref{thm:QE} and possibly working in $(T_\Mor)_\delta^*$, we may suppose $T_\delta^*$ eliminates $\bF$-quantifiers. Let $\varphi$ be an $\cL_\delta$-sentence. We may suppose $\varphi$ has no field-sort variable. Therefore, by Lemma-Definition \ref{lem:etiole}, $\varphi$ is equivalent, modulo $T_\delta^*$, to an $\cL$-sentence $\varphi^*$. The result follows from the completeness of $T$.  
\end{proof}

\begin{corollary}\label{cor:eq} For every $\cL_\delta$-formula $\varphi(x,z)$, with $x$ a tuple of $\bF$-variables and $z$ a tuple of auxiliary sort variables, there are a tuple $\bar{m}=(m_1,\ldots, m_{\ell(x)})$ with $m_i\in \N$, a tuple of $\bF$-variables $\bar{x}=(\bar{x}_1,\ldots, \bar{x}_{\ell(x)})$ with $\ell(\bar{x}_i)=m_i+1$ and an $\cL$-formula $\tilde{\varphi}(\bar{x},z)$ such that 
\[
T_\delta^* \models (\forall x)(\forall z)(\varphi(x,z)\leftrightarrow \tilde{\varphi}(\bd^{\bar{m}}(x),z)). 
\]

\end{corollary}
\begin{proof} Since $\varphi(x,z)$ is also an $(\cL_\Mor)_\delta$-formula, by Theorem \ref{thm:QE}, $\varphi$ is equivalent modulo $(T_\Mor)_\delta^*$ to an $\bF$-quantifier-free $(\cL_\Mor)_\delta$-formula $\psi(x,z)$. By Lemma-Definition \ref{lem:etiole}, there are $\bar{m}$, $\bar{x}$ and an $\cL_\Mor$-formula $\psi^*(\bar{x},z)$ such that $\psi(x,z)$ is equivalent modulo $(T_\Mor)_\delta^*$ to $\psi^*(\bd^{\bar{m}}(x),z)$. Letting $\tilde{\varphi}(\bar{x},z)$ be an $\cL$-formula equivalent modulo $T_\Mor$ to $\psi^*(\bar{x},z)$ 
completes the result.  
\end{proof}

\begin{corollary} Every $\cL_\delta$-definable set $X\subseteq K^n$ is of the form $\J_{\bar{m}}^{-1}(Y)$ for some tuple $\bar{m}$ and an $\cL$-definable set $Y\subseteq K^{\bar{m}+1}$.   \qed
\end{corollary}

By Corollary \ref{cor:eq}, the following notion of order is well-defined. 

\begin{definition}[Order]\label{def:order} Let $\varphi(x,z)$ be an $\cL_\delta$-formula with $x$ a tuple of $\bF$-variables and $z$ a tuple of auxiliary sort variables. Let $X$  be any $\cL_\delta$-definable set. We endow $\N^{\ell(x)}$ with the lexicographic ordering. 
\begin{itemize}
    \item The  \emph{order of $\varphi$} (as an element of $\N^{\ell(x)}$) is defined as 
    \[
    \ord(\varphi) \coloneqq \min_{\lex}\left\{\bar{m}\in \N^{\ell(x)} :\begin{array}{l} \text{ $\varphi$ is equivalent to $\psi(\bd^{\bar m}(x),z)$ modulo $T_\delta^*$} \\
    \text{with $\psi$ an $\cL$-formula} \end{array}\right\}. 
    \]
    \item The \emph{order of $X$}, is defined as 
\[
\ord(X) \coloneqq \min_{\lex}\{\ord(\varphi) : \text{$\varphi$ is an $\cL_\delta$-formula and $\varphi(\cK)=X$}\}. 
\]
\end{itemize}
\end{definition}

\begin{remark}\label{rem:order_up}
Let $X\subseteq \cK^{(x,z)}$ be an $\cL_\delta$-definable set where $x$ a tuple of $\bF$-variables and $z$ a tuple of auxiliary sort variables. It holds that $\ord(X)=\bar 0$ if and only if $X$ is $\cL$-definable. 
\end{remark}

Let us end this section by pointing out some final consequences of relative quantifier elimination on definable sets. The corollary below shows that the derivation introduces no new structure on auxiliary sorts.  

\begin{corollary}\label{cor:induced} Let $\cK$ be a model of $T_\delta^*$ and 
$z$ be tuple of auxiliary sort variables. Then every $\cL_\delta$-definable subset $X$ of $\cK^z$ is $\cL$-definable. 
\end{corollary}

\begin{proof} Let $\varphi(x,w,z)$ be an $\cL_\delta$-formula with $x$ a tuple of $\bF$-variables and $w$ be a tuple of auxiliary variables, such that $X=\varphi(a,b,\cK)$ for some $a\in \cK^x$ and $b\in \cK^w$. By Corollary \ref{cor:eq}, there are a tuple $\bar{m}$ and an $\cL$-formula $\tilde{\varphi}(\bar{x}, z)$ such that $\varphi(x,z)$ is equivalent to $\tilde{\varphi}(\bd^{\bar{m}}(x),z)$. This shows that $X=\tilde{\varphi}(\bd^{\bar{m}}(a),b,\cK)$, and hence $X$ is $\cL$-definable. 
\end{proof}

Somewhat orthogonal to the previous corollary, one can show that an $\cL_\delta$-definable family of definable subsets of the field sort with finite fibers, parametrized by auxiliary sorts, is already $\cL$-definable. We will delay the proof of that result (Lemma \ref{prop:fini}) until section \ref{sec:EI} since we will need it there.

\

By classical arguments from \cite{vandendries1989}, possibly working in $(T_\Mor)_\delta^*$, elimination of $\bF$-quantifiers also yields the following result (known in the one-sorted case in \cite{GP12}). 

\begin{corollary}[{\cite[Corollary 3.10]{GP12}}]\label{thm:delta-dim} The $\cL_\delta$-definable subsets of the field sort can be endowed with a dimension function (as defined by van den Dries in \cite{vandendries1989}). \qed
\end{corollary}

From Corollary \ref{cor:eq}, one may deduce some transfer results, like $\NIP$, distality or elimination of $\exists^{\infty}$ that will be stated and proven in the Appendix. 

It is worthy to mention that other model-theoretic properties such as the existence of prime models or dp-minimality do not transfer from $T$ to $T_\delta^*$. Indeed, M. Singer showed in \cite{singer1978} (see also \cite{Fr}) that $\CODF$ has no prime models (while $\RCF$ has) and Q. Brouette showed in his thesis \cite{brouette2015} that $\CODF$ is not dp-minimal (whereas $\RCF$ is). 
 
\section{Open core and transfer of elimination of imaginaries}\label{sec:OP-EI} 
 
\subsection{Open core}\label{sec:opencore}

Throughout this section we let $T$ be an open $\cL$-theory of topological fields and assume $T_\delta^*$ is consistent. We let $\cK$ be a model of $T$. Let us start by recalling the definition of open core. 

In order to make notation lighter on indices, for the remaining of this section, given an $\cL_\delta$-definable set $X\subseteq K^n$, we note $\ord(X)$ by $o(X)$.  

\begin{definition}\label{def:open_core} Let $\tilde \cL$ be an extension of $\cL$ and $\tilde \cK$ be an $\tilde \cL$-expansion of $\cK$. We say $\tilde \cK$ has \emph{$\cL$-open core} if, for every $n\geqslant 1$, every $\tilde \cL$-definable open subset of $K^n$ is $\cL$-definable. An $\tilde{\cL}$-theory $\tilde T$ extending $T$ has $\cL$-open core if every model of $\tilde T$ has $\cL$-open core. 
\end{definition}

Before proving the main results of this section, we give a useful characterization of the $\cL$-open core for $T_\delta^*$ (which will also be used in our transfer result of elimination of imaginaries). 

\begin{definition}\label{def:circ1} Let $X \subseteq K^n$ be a non-empty $\cL_\delta$-definable set. Given a tuple $\bar m\in \N^n$ and an $\cL$-definable set~$Z\subseteq K^{\bar m+1}$, we call the triple $(X, Z, \bar m)$ a \emph{linked triple} if 
\begin{enumerate}
\item $X=\J_{\bar m}^{-1}(Z)$ and
\item $\overline{Z}=\overline{\J_{\bar m}(X)}$.
\end{enumerate} 
\end{definition}

\begin{notation} For tuples $\bar{m}, \bar{d}\in \N^n$, we write $\bar{m}\leqslant \bar{d}$ to express that $m_i\leqslant d_i$ for all $1\leqslant i\leqslant n$. Assuming $\bar m\leqslant \bar d$ we let $\pi^{\bar d}_{\bar m}\colon K^{\bar d+1}\to K^{\bar m+1}$ be the projection sending each block of $d_{i}+1$ coordinates to its block of first $m_{i}+1$ coordinates, $1\leqslant i\leqslant n$, that is, 
\[
\pi_{\bar m}^{\bar d}(x_{1,0},\ldots,x_{1,d_{1}},  \ldots, x_{n,0},\ldots,x_{n,d_{n}}) = (x_{1,0},\ldots,x_{1,m_{1}},\ldots, x_{n,0},\ldots,x_{n,m_{n}}). 
\] 
\end{notation}
\begin{proposition}\label{thm:fermeture} The theory $T_{\delta}^*$ has $\cL$-open core if and only if for every $n\geqslant 1$ and every $\cL_\delta$-definable set $X\subseteq K^n$, there is a linked triple $(X,Z,\bar m)$ for some $\bar m\geqslant o(X)$. 
\end{proposition}

\begin{proof} 
$(\Rightarrow)$ Let~$Y\subseteq K^{o(X)+1}$ be an $\cL$-definable set such that $X=\J_{o(X)}^{-1}(Y)$. The subset $\overline{\J_{o(X)}(X)}$ is both closed and $\cL_{\delta}$-definable, so it is $\cL$-definable since $T_\delta^*$ as $\cL$-open core. Consider the $\cL$-definable set $Z\coloneqq Y\cap \overline{\J_{o(X)}(X)}$.  Since $\J_{o(X)}(X)\subseteq Z\subseteq \overline{\J_{o(X)}(X)}$, both properties (1) and (2) are easily shown.

$(\Leftarrow)$ It suffices to show that $\overline{X}$ is $\cL$-definable. By assumption there is a tuple $\bar m\in \N^n$ and an $\cL$-definable set $Z$ such that $(X, Z, \bar m)$ is a linked triple. We leave as an exercise to show that $\overline{X}=\overline{\pi_{\bar 0}^{\bar m}(\overline{Z})}$. The result follows since, as $Z$ is $\cL$-definable, so is $\overline{\pi_{\bar 0}^{\bar m}(\overline{Z})}$. 
\end{proof}

In the next lemmas, we relate the existence of linked triples for an $\cL_{\delta}$-definable set $X$ when $\bar m$ varies. We will use them in the proof of the $\cL$-open core.

\begin{lemma}\label{lem:cloture1} Let $X \subseteq K^n$ be a non-empty $\cL_\delta$-definable set. Let $\bar m, \bar d\in \N^n$ be such that $o(X) \leqslant \bar m <\bar d$ and suppose
$(X, Z, \bar d)$ is a linked triple. Then, $\overline{\J_{\bar m}(X)}=\overline{\pi_{\bar m}^{\bar d}(\overline{Z})}$. In particular, $\overline{\J_{\bar m}(X)}$ is $\cL$-definable. 
\end{lemma}

\begin{proof}
From left-to-right, let $a\in \J_{\bar m}(X)$ and $U$ an open neighbourhood of $a$. So $a=\bd^{\bar m}(x)$ with $x\in X$. Then by assumption $\bd^{\bar d}(x)\in Z$ and $a\in \pi_{\bar m}^{\bar d}(Z)$. 

Conversely, let $b\in \pi_{\bar m}^{\bar d}(\overline{Z})$ and $O$ an open neighbourhood of $b$. Then for some $z\in \overline{Z}$, $b=\pi_{\bar m}^{\bar d}(z)$. Let $\tilde O$ an open neighbourhood of $z$ such that $\pi^{\bar d}_{\bar m}(\tilde O)=O$. Since $z\in \overline{Z}$ and $\overline Z=\overline{\J_{\bar d}(X)}$, there is $x\in X$ such that $\bd^{\bar d}(x)\in \tilde O$ and so $\pi_{\bar m}^{\bar d}(\bd^{\bar d}(x))\in O$, namely $\bd^{\bar m}(x)\in O$. So $b\in \overline{\J_{\bar m}(X)}$.
\end{proof}

\begin{lemma}\label{lem:cloture2} Let $X \subseteq K^n$ be a non-empty $\cL_\delta$-definable set such that for some $\bar m\geq o(X)$, $\overline{\J_{\bar m}(X)}$ is $\cL$-definable. Then there is a linked triple of the form
$(X, Z, \bar m)$. 
\end{lemma}
\begin{proof} Let $Y\subseteq K^{o(X)+1}$ be an $\cL$-definable set such that $X=\J_{o(X)}^{-1}(Y)$. Set 
\[
Z \coloneqq \{x\in  \overline{\J_{\bar m}(X)} : \pi_{o(X)}^{\bar m} (x) \in Y \}. 
\]
The hypothesis yields that $Z$ is $\cL$-definable. It is an easy exercise to show that $(X,Z,\bar m)$ is a linked triple.
\end{proof}

\begin{lemma}\label{lem:up-down} Let $X \subseteq K^n$ be a non-empty $\cL_\delta$-definable set. Let $\bar m, \bar d\in \N^n$ be such that $o(X) \leqslant \bar m <\bar d$ and suppose
$(X, Z, \bar d)$ is a linked triple. Then, there is a linked triple of the form 
$(X,Z',\bar m)$.
\end{lemma}

\begin{proof} By Lemma \ref{lem:cloture1}, $\overline{\J_{\bar m}(X)}$ is $\cL$-definable. Therefore by Lemma \ref{lem:cloture2}, there is a linked triple of the form $(X,Z',\bar m)$.
\end{proof}

We now prove a rather technical lemma, which is a parametric version of the density of differential points (Lemma \ref{fact:density}).  But before doing that we need to put the $\bF$-quantifier free $\cL_{\delta}$-definable sets in an amenable form (in order to use the scheme (DL)).

\begin{lemma}\label{lem:deltaniceform} 
Let $\varphi(x,y)$ be an $\cL_\delta(\cK)$-formula where $x=(x_1,\ldots, x_{n})$ are $\bF$-variables and $y$ is a single $\bF$-variable. Let $\ord(\varphi)=(m_1,\ldots, m_{n+1})$ and $m$ be an integer such that $m\geqslant m_i$ for all $1\leqslant i\leqslant n+1$. Then, $\varphi$ is equivalent to a finite disjunction of $\cL_\delta(\cK)$-formulas of the form 
\[
\cZ_{\cA}^{S}(x,y) \wedge \theta(\bd^{m}(x), \bd^{m}(y)), 
\]
where $\theta$ is an $\cL(\cK)$-formula which defines an open subset of $K^{(n+1)(m+1)}$ and either $\cA\subseteq K\{x\}$ or $\cA$ contains only one differential polynomial $P$ of non-negative order in $y$ and $s_P$ divides $S$. In addition, $\ord_{x_i}(\cA,S)\leqslant 2m$ for $1\leqslant i\leqslant n$ and $\ord_y(\cA,S)\leqslant m$.   
\end{lemma}

\begin{proof} 
Possibly by adding formulas of the form $\delta^k(x_i)=\delta^k(x_i)$ (resp. $\delta^k(y)=\delta^k(y)$), by Corollary \ref{cor:eq} and condition \ref{condA}, $\varphi$ is equivalent to a disjunction of the form 
\[
Z_\cA(\bd^{m}(x),\bd^{m}(y))\wedge \theta(\bd^{m}(x), \bd^m(y)),
\]
where $\theta$ is an $\cL(\cK)$-formula which defines an open subset of $K^{(n+1)(m+1)}$ and $\cA\subseteq K[\bar{x},\bar{y}]$. Define 
\[
\cA'\coloneqq \{Q(\bd^{m}(x),\bd^{m}(y))  : Q\in \cA\}. 
\]
By definition, we have that $\varphi$ is equivalent to the corresponding disjunction of $\cL_\delta(\cK)$-formulas of the form 
\[
\cZ_{\cA'}(x,y)\wedge \theta(\bd^{{m}}(x), \bd^m(y)). 
\]
By Lemma \ref{lem:minord3}, the formula $\cZ_{\cA'}(x,y)$ is equivalent (modulo $T_\delta$) to a finite disjunction of formulas of the form $\cZ_\cB^{S_\cB}(x,y)$ 
such that $\ord_y(S_\cB)\leqslant m$, $\ord_{x_i}(\cB,S_\cB)\leqslant 2m$ for each $1\leqslant i\leqslant n$, and either 
\begin{enumerate}
\item[$\bullet$] $\cB\subseteq K\{x\}$ or
\item[$\bullet$] there is a unique $P\in \cB$ of non-negative order in $y$, $\ord_y(P_\cB)\leqslant m$ and $s_{P_\cB}$ divides $S_\cB$. 
\end{enumerate}
Then $\varphi(x,y)$ is equivalent to the disjunction of the corresponding disjunction of $\cL_\delta(\cK)$-formulas
\[
\cZ_\cB^{S_\cB}(x,y)\wedge \theta(\bd^{{m}}(x), \bd^m(y)). \qedhere
\]
\end{proof}

\begin{lemma}\label{lem:envelop} Let $X$ be an $\cL_{\delta}$-definable subset of $K^{n+1}$ with $\ord(X)=(m_1,\ldots,m_{n+1})$. Let $m\geqslant m_i$, $1\leqslant i\leqslant n+1$. 
Then, for $d=3m$ and $\bar d=(d,\ldots,d,m)\in \N^{n+1}$, there is an $\cL$-definable subset $Y\subseteq K^{\bar d+1}$ such that 
\begin{enumerate}
\item $X=\J_{\bar d}^{-1}(Y)$ and 
\item for every $a\in K^n$ and $c\in K^{m+1}$ such that $(\bd^d(a),c)\in Y$ it holds that for every open neighbourhood $W$ of $c$ there is $b\in K$ such that $\bd^m(b)\in W$ and $(\bd^d(a),\bd^m(b))\in Y$. 
\end{enumerate}
In particular, $\vert X_{a}\vert=\vert Y_{\bd^d(a)}\vert$ for every $a\in K^n$ such that $X_{a}$ is finite.
\end{lemma}

\begin{proof} 
Let $\varphi(x,y)$ be an $\cL_\delta(\cK)$-formula where $x=(x_1,\ldots, x_n)$, $y$ is a single variable and $\varphi(\cK)=X$. By Lemma \ref{lem:deltaniceform}, $\varphi(x,y)$ is equivalent, modulo $T_\delta^*$, to a finite disjunction of the form 
\[
\bigvee_{j\in J} \cZ_{\cA_j}^{S_j}(x,y) \wedge \theta_j(\bd^{m}(x), \bd^m(y)), 
\]
where for each $j\in J$, $\theta_j$ is an $\cL(\cK)$-formula which defines an open subset of $K^{(n+1)(m+1)}$ and either $\cA_j\subseteq K\{x\}$ or $\cA_j\subseteq K\{x,y\}$, it only contains one differential polynomial $P_j$ of non-negative order $k_j\leqslant m$ in $y$ and $s_{P_j}$ divides $S_j$. In addition, $\ord_{x_i}(\cA_j,S_j)\leqslant 2m$ for $1\leqslant i\leqslant n$ and $\ord_y(\cA_j,S_j)\leqslant m$. For each $j\in J$, let $\tilde{\theta}_j(\bar{x}[m], \bar{y})$ be the $\cL(\cK)$-formula $\theta_j(\bar{x},\bar{y})\wedge S_j^*(\bar{x}[m],\bar{y})\neq 0$, where $\bar{x}\coloneqq(\bar{x}_1,\ldots,\bar{x}_n)$ with $\ell(\bar{x}_i)=m+1$ and $\ell(\bar y)=m+1$. Note that $\tilde{\theta}_j$ defines an open subset of $K^{n(2m+1)}\times K^{m+1}$. For each $j\in J$, we define by cases an $\cL(\cK)$-formula $\psi_j(\bar{x}[2m],\bar{y})$ depending on whether $\cA\subseteq K\{x\}$ or not: 
\begin{enumerate}
\item[$(i)$] if $\cA_{j}\subseteq K\{x\}$ then  $\psi_j(\bar{x}[2m],\bar{y})$ is 
\[
Z_{\cA_{j}}(\bar x[m], \bar{y})\wedge \tilde{\theta}_j(\bar{x}[m],\bar{y}).
\] 
\item[$(ii)$] otherwise, 
we define $\psi_j(\bar{x}[2m],\bar{y})$ as 
\[
Z_{\cA_{j}}(\bar x[m], \bar{y})\wedge \tilde{\theta}_{j}(\bar{x}[m],\bar{y})\wedge\lambda_{P_{j}}^{m-k_j}(\bar{x}[m][m-k_j],\bar{y}). 
\] 
\end{enumerate}
Let $\psi(\bar{x}[2m],\bar{y})$ be the disjunction $\bigvee_{j\in J} \psi_j(\bar{x}[2m],\bar{y})$ and $Y$ be the subset of $K^{\bar{d}+1}$ defined by $\psi$. Let us show (1). The inclusion $\J_{\bar d}^{-1}(Y)\subseteq X$ is clear. The converse follows by noting that for each $j\in J$ for which $\psi_j$ is as in $(ii)$
\[
T_\delta^* \models \forall x \forall y (\cZ_{\cA_j}^{S_j}(x,y) \to  \lambda_{P_{j}}^{m-k_j}(\bd^{3m-k_j}(x), \bd^m(y))).  
\]
It remains to show (2). 

Fix $a\in K^n$ and $c=(c_0,\ldots,c_m)\in K^{m+1}$ such that $(\bd^d(a),c)\in Y$. Let $j\in J$ be such that $\psi_j(\bd^d(a),c)$ holds.
We split in cases. If $\psi_j$ is as in $(i)$, then the result follows from Lemma \ref{fact:density}. So suppose $\psi_j$ is as in $(ii)$. Let $W$ be an open neighbourhood of $(c_0,\ldots,c_{m})$. Without loss of generality, we may suppose there is $V$ an open neighbourhood of $\bd^{2m}(a)$ such that $V\times W\subseteq \tilde\theta_j(\cK)$. Let $V\times V_1$ be an open neighbourhood of $\bd^{3m-k_j}(a)$. By the continuity of the functions $(f_{i}^{P_{\ell}})^*$ (see Lemma-Definition \ref{lemdef:ratioprolong}), we may shrink $V\times V_1$ to a smaller open neighbourhood of $\bd^{3m-k_j}(a)$ and find an open neighbourhood $W_{1}$ of $(c_{0},\ldots,c_{k_{j}})$ such that, letting $U\coloneqq V\times V_1\times W_{1}$ 
\[
W_1\times (f_1^{P_j})^*(U) \times \ldots \times (f_{m-k_j}^{P_j})^*(U) \subseteq W, 
\] 
where $(f_i^{P_j})^*$, $1\leqslant i\leqslant m-k_j$ is seen as a function from $K^{n(d+1)}\times K^{k_j+1}$ to $K$. By the scheme (DL), we can find a differential tuple $\bd^{k_j}(b)\in W_{1}$ such that 
\[
K\models P_{j}^*(\bd^{2m}(a),\bd^{k_j}(b))=0\wedge\frac{\partial}{\partial y_{k_j}} P_{j}^*(\bd^{2m}(a),\bd^{k_j}(b))\neq 0.
\]
This implies that $\delta^{k_j+i}(b)=(f_{i}^{P_{j}})^*(\bd^{2m+i}(a),\bar \delta^{k_j}(b))$ for each $i\in \{1,\ldots, m-k_{j}\}$, and hence $\bd^m(b)\in W$. We obtain thus that $(\bd^{2m}(a), \bd^{m}(b))\in V\times W\subseteq \tilde \theta_j(\cK)$. This shows finally that $\psi_j(\bd^{d}(a),\bd^m(b))$ holds, so $(\bd^d(a),\bd^m(b))\in Y$. 

The last statement follows directly from part (2) and the fact the topology is Hausdorff. 
\end{proof}

\begin{corollary}\label{cor:envelop1} Let $X\subseteq K$ be $\cL_{\delta}$-definable. Then for every $m\geqslant o(X)$, $X$ has a linked triple of the form $(X,Z,m)$. \qed
\end{corollary}


\begin{theorem}\label{thm:opencoregen} Let $T$ be an open $\cL$-theory of topological fields. Then the theory $T_\delta^*$ has $\cL$-open core.
\end{theorem}

\begin{proof} 
Let $X$ be an $\cL_{\delta}$-definable subset of $K^{n+1}$. By Proposition \ref{thm:fermeture}, it suffices to show that $X$ has 
a linked triple, for some $\bar m\geq o(X)$. However, we will show by induction on $n$ the following stronger statement:
\begin{equation}\label{eq:IH}\tag{$\ast$}
\text{ for every $\bar m\geq o(X)$, $X$ has a linked triple $(X, Z, \bar m)$. }
\end{equation} 

For $n=0$, \eqref{eq:IH} follows from Corollary \ref{cor:envelop1}. Suppose $n>0$. By Lemma \ref{lem:up-down}, it suffices to show that $X$ has a linked triple $(X,Z,\bar m)$ for arbitrarily large $\bar m\geq o(X)$.

Consider $\pi(X)\subseteq K^n$ where $\pi\colon K^{n+1}\to K^n$ is the projection onto the first $n$-coordinates. Let $\bar m=(m_1,\ldots,m_{n+1})$ and suppose that $\bar m\geq o(X)$ and $(m_1,\ldots,m_n)\geq o(\pi(X))$. We will show that $X$ has a linked triple of the form $(X, Z, \bar d)$, with $\bar d$ chosen as follows:
for $m\geq m_i$, $1\leqslant i\leqslant n+1$, let $d=3m$, $\overline{3m}=(3m,\dots,3m)\in \N^n$ and $\bar d=(\overline{3m}, m)$.

By Lemma \ref{lem:envelop}, there is an $\cL$-definable subset $Y\subseteq K^{\bar d+1}$ such that $X=\J_{\bar d}^{-1}(Y)$ and, if $x=\bd^d(a)$ for $a\in K^n$ and $(\bd^d(a),c)\in Y$ for some $c\in K^{m+1}$, given any open neighbourhood $W$ of $c$, there is $b\in K$ such that $\bd^m(b)\in W$ and $(\bd^d(a),\bd^m(b))\in Y$.

By induction hypothesis, $\pi(X)$ has a linked triple of the form $(\pi(X), Z, \overline{3m})$.
Set $Y'=\{(x,y) \in Y: x\in Z\}$. We claim that $(X,Y', \bar d)$ is a linked triple for $X$. First note that $\J_{\bar d}^{-1}(Y')=X$, and therefore, $\overline{\nabla_{\bar d}(X)}\subseteq \overline{Y'}$. 

To show $\overline{Y'} \subseteq \overline{\nabla_{\bar d}(X)}$, let $(a,b)\in Y'$ and $U\times V$ be an open neighbourhood of $(a,b)$. By cell decomposition (Theorem \ref{thm:newCD}), $Y'=\bigcup_{i\in I} Y_i$, where $Y_i$ is a cell and $I$ is a finite set. Let $i\in I$ be such that $(a,b)\in Y_i$. By Lemma \ref{lem:Ldens}, there is an open neighbourhood $U_1\subseteq U$ of $a$ such that for every $a'\in U_1$, there is $b'\in V$ such that $(a',b')\in Y_i$. Since $a'\in Z$ and $(\pi(X),Z,\overline{3m})$ is a linked triple, there is $c\in \pi(X)$ such that $\bd^d(c) \in U_1$. Let $b'\in V$ be such that $(\bd^d(c), b')\in Y_i$. Since $Y_i\subseteq Y$, there is $c'\in K$ such that $\bd^m(c')\in V$ and $(\bd^d(c), \bd^m(c'))\in Y$. Therefore $(c,c')\in X$ and $(\bd^d(c), \bd^m(c'))\in U\times V$, which shows that $\overline{Y'}\subseteq\overline{\J_{\bar d}(X)}$. 
\end{proof}

\begin{corollary}\label{cor:opencore}
Let $T$ be one of the following theories: $\RCF$, $\ACVF_{0,p}, \RCVF, \PCF$ or, in general the $\cL_\Div$-theory of a henselian valued field of characteristic 0. Then, $T_\delta^*$ has $\cL$-open core.
\end{corollary}
\begin{proof} All these theories are open $\cL$-theories of topological fields (see Examples \ref{examples}) and by Theorem \ref{thm:consistency} $T_\delta^*$ is consistent. 
\end{proof}

\subsection{$\delta$-Cell decomposition}\label{sec:delta-cell-decomp}

Let us start by defining $\delta$-cells. 

\begin{definition}[$\delta$-cells]\label{def:delta-cell} A subset $X\subseteq K^n$ is a $\delta$-cell if there are $\bar d\in \N^n$ and an $\cL$-definable cell $Y$ (as defined in Definition \ref{def:cells}) such that $X=\J_{\bar d}^{-1}(Y)$ and $\overline{\J_{\bar d}(X)}=\overline{Y}$.  
\end{definition}

\begin{theorem}[$\delta$-Cell decomposition]\label{delta-cell-decom} Let $\cK$ be a model of $T_{\delta}^*$. Then, for every $n\geqslant 1$, every $\cL_\delta$-definable set $X\subseteq K^n$ is a disjoint union of finitely many $\delta$-cells. 
\end{theorem}

\begin{proof} By Theorem \ref{thm:opencoregen}, $T_\delta^*$ has $\cL$-open core.
By Proposition \ref{thm:fermeture} and Lemma \ref{lem:up-down}, we may associate to $X$ a linked triple $(X, Y, \bar d)$ with $\bar d=o(X)=(d_1,\ldots,d_n)$. We proceed by induction on $\dim(Y)$. 
If $\dim(Y)=0$, then $Y$ is a cell and hence $X$ is already a $\delta$-cell. Suppose $\dim(Y)>0$. 

By cell decomposition (Theorem \ref{thm:newCD}), $Y$ can be expressed as a finite disjoint union of cells $Z_{i}$, $i\in I$. Let $I_{0}\coloneqq \{i\in I\colon \dim(Z_{i})=\dim(Y)\}$ and let $I_{1}\coloneqq I\setminus I_{0}$. Let $X_{1}\coloneqq \J_{\bar d}^{-1}(\bigcup_{i\in I_{1}} Z_{i})$ and for $i\in I_{0}$, let $X_{i}\coloneqq \J_{\bar d}^{-1}(Z_{i})$. We show the result for $X_1$ and $X_i$ for $i\in I_0$. 

For $X_1$, we apply the induction hypothesis. Indeed, by the $\cL$-open core, for $Y'=\overline{\J_{\bar d}(X_1)}\cap \bigcup_{i\in I_1} Z_i$, $(X_{1},Y',\bar d)$ is a linked triple. In addition, $\dim(Y')\leqslant\dim(\overline{\J_{\bar d}(X_1)}) \leqslant \dim(\overline{\bigcup_{i\in I_{1}} Z_{i}})<\dim(Y)$, and the result follows by induction. 

For $j\in I_{0}$, let us show that $X_{j}$ is already a $\delta$-cell. It amounts to show that $\overline{\J_{\bar d}(X_{j})}=\overline{Z_{j}}$. One inclusion is clear, so let us show that $\overline{Z_{j}}\subseteq \overline{\J_{\bar d}(X_{j})}$. It suffices to show that for $z\in Z_{j}$, every open neighbourhood of $z$ intersects $\J_{\bar d}(X_{j})$. First let us show the following claim.
\begin{claim}\label{claim:density_new}
Let $z\in Z_j$, then for every open neighbourhood $U$ of $z$, there is $x\in X$ such that $\bd^{\bar d}(x)\in \overline{Z_j}\cap U$.
\end{claim}
Suppose for a contradiction that there is an open neighbourhood $U$ of $z$ such that $U\cap \overline{Z_{j}}\cap\J_{\bar d}(X)=\emptyset$. We split in two cases:

\emph{Case 1:} Suppose that $U\cap Z_{j}\subseteq \bigcup_{i\in I\setminus \{j\}} \overline{Z_{i}}$. Since the $(Z_i)_{i\in I}$ are disjoint, we have $U\cap Z_{j}\subseteq \bigcup_{i\in I\setminus \{j\}} \overline{Z_{i}}\setminus Z_{i}$. This contradicts that $\dim(U\cap Z_{j})=\dim(Z_{j})=\dim(Y)$ and $\dim(\bigcup_{i\in I} \overline{Z_{i}}\setminus Z_{i})<\dim(Y)$.

\medskip

\emph{Case 2:} Suppose there is $z'\in U\cap Z_{j}\setminus(\bigcup_{i\in I\setminus \{j\}} \overline{Z_{i}})$. Let $U'\subseteq U$ be an open neighbourhood of $z'$ such that $U'\cap \overline{Z_i}=\emptyset$  for every $i\in I\setminus \{j\}$. Since $z'\in Z_j$ and $(X,Y, \bar d)$ is a linked triple, there is $x\in X$ such that $\bd^{\bar d}(x)\in U'\cap \overline{Y}$. By the choice of $U'$, $\bd^{\bar d}(x)\in \overline{Z_j}$, but $U'\subseteq U$, so we contradict the assumption on $U$. This completes the claim. 

\medskip

Let $\rho_{Z_j}$ be the projection associated to the cell $Z_j$ such that $Z_j$ is the graph along $\rho_{Z_j}$ of a continuous correspondence $f_{Z_j}\colon \rho_{Z_j}(Z_j) \to K^{(\bar{d}+1)-\dim(Z_j)}$ (see Definition \ref{def:cells}). To conclude, let $U$ be an open neighbourhood of $z$. Without loss, we may assume that $\rho_{Z_j}(U)\subseteq \rho_{Z_j}(Z_j)$. By Claim \ref{claim:density_new}, $U$ contains an element $\bd^{\bar d}(x)\in  \overline{Z_{j}}$ with $x\in X$. Let us show that $\bd^{\bar d}(x)\in Z_{j}$. Note that if $\rho_{Z_j}$ is the identity, the result follows directly. For notational simplicity, suppose $\rho_{Z_j}$ is onto the first $m$-coordinates for some $m<\sum_{i=1}^n(d_i+1)$. Recall that $\rho_{Z_j}^{\perp}$ denotes the complement projection to $\rho_{Z_j}$. Let $w=\rho_{Z_j}(\bd^{\bar d}(x))$. Since $\rho_{Z_j}(U)\subseteq \rho_{Z_{j}}(Z_{j})$, it suffices to show that $\rho_{Z_j}^\perp(\bd^{\bar d}(x))\in f_{Z_j}(w)$. For suppose not. Then, since the topology is Hausdorff, there are disjoint open sets $V_1$ and $V_2$ such that $V_1$ contains $\{(w, y) : y\in f_{Z_j}(w)\}$ and $V_2$ contains $\bd^{\bar d}(x)$. Since $w\in \rho_{Z_j}(V_1)\cap \rho_{Z_j}(V_2)$, we may suppose without loss that $\rho_{Z_j}(V_1)= \rho_{Z_j}(V_2)$. This contradicts the continuity of $f_{Z_j}$ and the fact that $\bd^{\bar d}(x)\in \overline{Z_j}$. 
\end{proof}

\begin{remark}\label{rem:delta-CD} In \cite{BMR}, the authors proved a cell decomposition theorem for $\CODF$ (see \cite[Theorem 4.9]{BMR}). According to their definition, $\delta$-cells are $\cL_\delta$-definable sets of the form $X=\J_{\bar{d}}^{-1}(Y)$ where $Y$ an o-minimal cell (of $\RCF$). They do not require however the density condition $\overline{\J_{\bar{d}}(X)}=\overline{Y}$ present in Definition \ref{def:delta-cell}. Although our definition of $\cL$-definable cells does not coincide with o-minimal cells, o-minimal cells are cells in the sense of Definition \ref{def:cells}, so the proof of Theorem \ref{delta-cell-decom} applies word for word replacing our cells by o-minimal cells. 
\end{remark}

\subsection{Transfer of elimination of imaginaries} \label{sec:EI}

In this section, we show how to transfer elimination of imaginaries from $T$ to $T_\delta^*$. Our proof relies on properties of the topological dimension in models of $T$. The strategy we employ has been previously used by M. Tressl in the case of $\CODF$. For background facts on the elimination of imaginaries, we refer to \cite[Section 8.4]{tent-ziegler2012}. We will need the following two lemmas.

\begin{lemma}\label{lem:order_EI} Let $X\subseteq K^n$ be an $\cL_\delta$-definable set. Let $\bar m_1, \bar m_2\in \N^n$. If $\bar m_1\leqslant \bar m_2$ then
\[
\Dim(\overline{\J_{\bar m_1}(X)})\leqslant \Dim(\overline{\J_{\bar m_2}(X)}).
\]
\end{lemma}

\begin{proof} Let $\pi\colon \overline{\J_{\bar m_1}(X)}\to K^\ell$ be a projection such that $\pi(\overline{\J_{\bar m_1}(X)})$ has non-empty interior. Then, $\pi\circ\pi_{\bar{m}_1}^{\bar{m}_2}(\overline{\J_{\bar m_2}(X)})$ has non-empty interior. 
\end{proof}

\begin{lemma}\label{prop:fini} Let $X\subseteq \cK^{(x,z)}$ be a definable set where $x$ is a tuple of $\bF$-variables and $z$ is a tuple of auxiliary sort variables. Suppose that for every $b\in \cK^z$, the fiber $X_b$ is finite. Then $X$ is $\cL$-definable.   
\end{lemma}

The main ingredient of the proof of the above result is Lemma \ref{lem:envelop}. However it is a bit technical and so we chose to postpone it after stating (and proving) the main result of the section. 

\

Let $\G$ be a collection of sorts of $\cL^{\mathrm{eq}}$. We let $\cL^{\G}$ denote the restriction of $\cL^{\eq}$ to the field sort together with the new sorts in $\G$. The following proof will make use of properties (D1)-(D3) of the topological dimension listed in Proposition \ref{prop:conseAA2}). 

\begin{theorem}\label{thm:EI} Suppose that $T$ admits elimination of imaginaries in $\cL^\G$. Then the theory $T^*_{\delta}$ admits elimination of imaginaries in $\cL_{\delta}^\G$. 
\end{theorem}

\begin{proof} Fix a sufficiently saturated model $\cK$ of $T_\delta^*$. For the sake of exposition, we first assume $\cL$ is one-sorted and later explain how to modify the proof to obtain the general case. 

Let $X\subseteq K^n$ be a non-empty $\cL_{\delta}$-definable set. It suffices to show that $X$ has a code in $\cL_{\delta}^\G$ (that is, an element $e\in \G$ such that $\sigma(e)=e$ if and only if $\sigma(X)=X$ for every $\cL_\delta$-automorphism $\sigma$ of $\cK$). Observe that every $\cL$-definable set has a code in $\cL^\G$, and therefore a code in $\cL_\delta^\G$, as the $\cL_\delta$-automorphism group of $\cK$ is a subgroup of the $\cL$-automorphism group of $\cK$. Consider the set $\widetilde{X}\supseteq X$ defined by 
\[
\widetilde{X}\coloneqq \J_{o(X)}^{-1}(\overline{\J_{o(X)}(X)}).
\]
By Theorem \ref{thm:opencoregen}, $T_\delta^*$ has $\cL$-open core, so the set $\overline{\J_{o(X)}(X)}$ is $\cL$-definable. We proceed by induction on $\Dim(\overline{\J_{o(X)}(X)})$. If $\Dim(\overline{\J_{o(X)}(X)})= 0$, then $X$ is finite (by (D1)) and in particular $\cL$-definable, so it has a code in $\cL_{\delta}^\G$. To show the inductive step we need the following claim: 
\begin{claim}\label{cla:dim} $\Dim(\overline{\J_{o(X)}(\widetilde{X}\setminus X)}) < \Dim(\overline{\J_{o(X)}(X)})$.
\end{claim}
\noindent Suppose the claim holds. Since $o(\widetilde{X}\setminus X)\leqslant o(X)$, by Lemma \ref{lem:order_EI}, we have that 
\[
\Dim(\overline{\J_{o(\widetilde{X}\setminus X)}(\widetilde{X}\setminus X)}) \leqslant \Dim(\overline{\J_{o(X)}(\widetilde{X}\setminus X)}). 
\]
Therefore, by Claim \ref{cla:dim} and the induction hypothesis, let $e_{1}$ be a code for $\widetilde{X}\setminus X$. By the previous observation, let $e_{2}$ be a code for $\overline{\J_{o(X)}(X)}$ (which is $\cL$-definable by the $\cL$-open core hypothesis). It is an easy exercise to show that $e=(e_{1},e_{2})$ is a code for $X$.

It remains to prove the claim. By the $\cL$-open core property
and Proposition \ref{thm:fermeture}, let $(X, Z, o(X))$ be a linked triple. Applying (D3), we have 
\begin{align*}
\Dim(\overline{\J_{o(X)}(\widetilde{X}\setminus X)}) & =  \Dim(\overline{\J_{o(X)}(\J_{o(X)}^{-1}(\overline{\J_{o(X)}(X)})\setminus X)}) \\
								& =  \Dim(\overline{\J_{o(X)}(\{x\in K^n : \J_{o(X)}(x)\in \overline{Z}\}\setminus X)}) \\  
								& =  \Dim(\overline{\J_{o(X)}(\{x\in K^n : \J_{o(X)}(x)\in \overline{Z}\setminus Z\})}) \\
								& \leqslant \Dim(\overline{Z}\setminus Z) < \Dim(\overline{Z})= \Dim(\overline{\J_{o(X)}(X)}). 
\end{align*}
This completes the proof in the one-sorted case. 

For the general case, we need to code
a non-empty definable subset $X\subseteq K^n\times \cK^z$ where $z$ is a tuple of auxiliary sort variables. Note that if $n=0$ then, by Corollary \ref{cor:induced}, $X$ is already $\cL$-definable and has therefore a code. Thus we may suppose $n\geqslant 1$. We proceed analogously, using the following notion of dimension on general $\cL$-definable sets: for $X$ as above and $Z$ the projection of $X$ onto the $z$-variables, we set
\[
\dim(X) \coloneqq \max \{\dim(X_b) : b\in Z\}. 
\]
If $\dim(X)=0$, this means that for all $z\in Z$, the fiber $X_z$ is finite. Then, by Lemma \ref{prop:fini}, $X$ is $\cL$-definable and has a code. For the inductive step, define for any $\bar d\in \N^n$, the operations $\J_{\bar d}(X)$ (resp. for $\J_{\bar d}^{-1}$) and $\overline{X}$ 
fiber-wise, that is, 
\[
\J_{\bar d}(X) \coloneqq \bigcup_{b\in Z} \J_{\bar d}(X_b)\times\{b\} \hspace{1cm} \overline{X} \coloneqq \bigcup_{b\in Z} \overline{X_b}\times\{b\} . 
\]
The proof now follows exactly the same strategy as in the one-sorted case, after showing the corresponding version of Claim \ref{cla:dim} and property (D3) for this new notion of dimension. However, note that both properties follow easily since they hold fiber-wise. 
\end{proof}

\begin{corollary}\label{cor:EI_examples} Let $\G$ denote the geometric language of valued fields. The theories $(\ACVF_{0,p})_\delta^*$, $\RCVF_\delta^*$ and $\PCF_\delta^*$ have elimination of imaginaries in $\cL_\delta^\G$. 
\end{corollary} 

\begin{proof} Let $T$ be either $\ACVF_{0,p}$, $\RCVF$ or $\PCF$. The theory $T$ has elimination of imaginaries in $\cL^\G$ by results of Haskell, Hrushovski and Macpherson for $\ACVF$ \cite{HHM2006}, of Mellor for $\RCVF$ \cite{mellor2006}, and of Hrushovski, Martin and Rideau for $\PCF$ \cite{HMR2018}. By Corollary \ref{cor:opencore}, $T_\delta^*$ has $\cL$-open core. The result follows by Theorem \ref{thm:EI}. 
\end{proof}

{\bf Proof of Lemma \ref{prop:fini}.}
Let $\varphi(x,z)$ be an $\cL_\delta(\cK)$-formula defining $X$ and $x=(x_1,\ldots, x_{n+1})$ with $n\geqslant 0$. By Lemma \ref{lem:deltaniceform}, letting $\ord(\varphi)=(m_1,\ldots,m_{n+1})\in \N^{n+1}$ and for $m=\max\{m_i\colon 1\leqslant i\leqslant n+1\}$, the formula  $\varphi(x,z)$ is equivalent to a finite disjunction of $\cL_\delta(\cK)$-formulas of the form 
\begin{equation}\label{eq:formula}
\cZ_{\cA}^{S}(x) \wedge \theta(\bd^{m}(x),z),     
\end{equation}
where $\theta$ is an $\cL(\cK)$-formula such that for every each $b\in \cK^z$, $\theta(\cK,b)$ is an open set and either 
\begin{enumerate}
    \item $\cA\subseteq K\{x_1,\ldots, x_n\}$ or 
    \item $\cA$ contains only one differential polynomial $P$ with
    $\ord_{x_{n+1}}(P)=m_P\leqslant m$, $\ord_{x_j}(P)\leqslant 2m$, $1\leqslant j\leqslant n$ 
    and $s_P$ divides $S$. 
\end{enumerate}
Without loss of generality we may assume $\varphi(x,z)$ is already the formula given in \eqref{eq:formula}. In addition, since each $X_b$ is finite, case (1) above cannot hold, so we must be in case (2). 

By induction on $n$, we show there are differential polynomials $P_j\in K\{x_1,\ldots, x_j\}$ for $j\in\{1,\ldots, n+1\}$ such that 
\begin{equation}\label{eq:poly}
    T_\delta^*\models (\forall z)(\forall x)(\varphi(x, z)\rightarrow
\bigwedge_{j=1}^{n+1} P_j(x_1,\ldots x_j)=0\wedge s_{P_j}(x_1,\ldots,x_j)\neq 0),
\end{equation}
and an $\cL(\cK)$-formula $\tilde{\varphi}(\bar x, z)$ where $\bar x=(\bar x_1,\cdots,\bar x_{n+1})$, $\bar{x}_j=(x_j,x_{j,1}, \ldots, x_{j, k_{j}})$ with $k_{j}\coloneqq \ord_{x_j}(P_j)$, $1\leqslant j\leqslant n+1$, such that
\begin{equation}\label{eq:formule2}
T_\delta^*\models (\forall z)(\forall \bar x)(\tilde{\varphi}(\bar x, z)\leftrightarrow (\varphi(x, z)\wedge \bigwedge_{j=1}^{n+1} \bd^{k_{j}}(x_j)=\bar{x}_j)).
\end{equation}
This implies the statement, since $X$ is defined by the $\cL(\cK)$-formula 
\[
(\exists x_{1,1})\cdots(\exists x_{1, k_{1}})\cdots(\exists x_{n+1,1})\ldots(\exists x_{n+1, k_{n+1}})\tilde{\varphi}(\bar{x},z).
\]
Suppose $n=0$. Set $P_1\in K\{x_1\}$ to be the differential polynomial $P$ from (2) above. The choice of $P_1$ ensures already \eqref{eq:poly}. Define $\tilde{\varphi}(\bar{x}_1,z)$ to be the $\cL(\cK)$-formula given by 
\begin{align*}
P_1^*(\bar{x}_1)=0\wedge S^*(g_1^{\ord(S)}(\bar{x}_1))\neq 0\wedge
\theta(g_1^{m}(\bar{x}_1),z),
\end{align*}
where for any $d\geqslant 0$, $g_1^d(\bar{x}_1)$ is the definable function given by the prolongation of length $d$ (possibly truncation depending on $d$) of $\bar{x}_1$ with respect to $P_1$, namely, 
\[
g_1^d(\bar{x}_1)=
\begin{cases}
(x_1,x_{1,1}, \ldots, x_{1,d}) & \text{ if $d\leqslant k_{1}$} \\
(\bar{x}_1,(f_1^{P_1})^*(\bar{x}_1),\ldots, (f_{d-k_{1}}^{P_1})^*(\bar{x}_1)) & \text{ if $d> k_{1}$,} 
\end{cases}
\]
where $f_i^{P_1}$, $i\geq 1$, is defined as in Lemma-Definition \ref{lemdef:ratioprolong}. In particular, if $P_1(x_1)=0$, $s_{P_1}(x_1)\neq 0$ and $\bar{x}_1=\bd^{k_1}(x_1)$, then $\bd^{d}(x_1)=g_1^d(\bar{x}_1)$ for every $d\geqslant 0$. Let us show $\tilde{\varphi}$ satisfies \eqref{eq:formule2}. The implication from right to left follows from \eqref{eq:poly} and the definition of the functions $g_1^d$. For the converse, let $\bar{a}_1=(a_1,a_{1,1},\ldots, a_{1, k_{1}})$ and suppose that $\tilde{\varphi}(\bar{a}_1,b)$ holds for $b\in \cK^{z}$. By the axiom scheme (DL) and similarly to the argument given in Lemma \ref{lem:envelop}, for any neighbourhood of $\bar{a}_1$ there is an element $u\in K$ such that 
\[
P(u)=0 \wedge S(u)\neq 0\wedge \theta(\bd^m(u),b)
\]
holds. Since the topology is Hausdorff and $X_b$ is finite, we must have that $\bar{a}_1$ is of the form $\bd^{k_{1}}(a_1)$ and therefore, by Lemma-Definition \ref{lemdef:ratioprolong}, $\varphi(a,b)$ holds. 

Suppose $n>0$. Consider the $\cL_\delta(\cK)$-formula $\xi(x_1,\ldots, x_n,z)$ given by $(\exists x_{n+1}) \varphi(x,z)$. Note that for each $b\in \cK^z$, the fiber $\xi(\cK,b)$ is finite. Therefore, by induction, there are polynomials $P_j\in K\{x_1,\dots, x_j\}$ for $1\leqslant j\leqslant n$ satisfying the corresponding $\eqref{eq:poly}$ with respect to $\xi$, and an $\cL(\cK)$-formula $\tilde{\xi}(\bar{x}_1,\ldots, \bar{x}_n,z)$ satisfying the corresponding \eqref{eq:formule2}. We set $P_{n+1}$ as the differential polynomial $P$ from (2) above. Since $\varphi(x,z)$ implies $\xi(x_1,\ldots, x_n,z)$, the choice of $P_{n+1}$ and the induction hypothesis ensure that \eqref{eq:poly} holds. Define the formula $\tilde{\varphi}(x,z)$ as  
\begin{align*}
& P_{n+1}^*(g_1^{\ord_{x_1}(P)}(\bar{x}_1), \ldots,  g_{n}^{\ord_{x_{n}}(P)}(\bar{x}_1, \ldots, \bar{x}_{n}), \bar{x}_{n+1})=0 \ \wedge \\ 
& S^*(g_1^{\ord_{x_1}(S)}(\bar{x}_1), \ldots,  g_{n+1}^{\ord_{x_{n+1}}(S)}(\bar{x}_1, \ldots, \bar{x}_{n+1}))\neq 0 \ \wedge \\  
& \theta(g_1^{m}(\bar{x}_1), \ldots, g_{n+1}^{m}(\bar{x}_1, \ldots, \bar{x}_{n+1}),z) \wedge \tilde{\xi}(\bar{x}_1,\ldots, \bar{x}_n, z)
\end{align*}
where for each $d\geqslant 0$ and each $1\leqslant j\leqslant n+1$, the function $g_j^d(\bar{x}_1,\ldots, \bar{x}_j)$ is inductively defined and corresponds to 
the prolongation of length $d$ (possibly truncation depending on $d$) of $\bar{x}_j$ with respect to $P_j$ and $g_1^{d'}(\bar{x}_1),\ldots, g_{j-1}^{d'}(\bar{x}_1, \ldots,\bar{x}_{j-1})$ for all $d'\geqslant 0$ (we omit the formulas, for simplicity). As before, for each $1\leqslant j\leqslant n+1$, if $\bigwedge_{i=1}^j \bd^{k_j}(x_j)= \bar{x}_j$, then $\bd^d(x_j) = g_j^d(\bar{x}_1,\ldots, \bar{x}_j)$ for all $d\geqslant 0$ (whenever $\bigwedge_{j=1}^{n+1} P_j(x_1,\ldots, x_j)=0\wedge s_{P_j}(x_1,\ldots,x_j)\neq 0$). It remains to show \eqref{eq:formule2}. The implication from right to left follows by induction using in addition \eqref{eq:poly} and the definition of $g_j^d$ for all $1\leqslant j\leqslant n+1$ (note that the $\cL_\delta(\cK)$-formula $\cZ_{\cA\setminus\{P\}}(x_1,\ldots, x_n)$ is implied by $\xi$). For the converse, suppose that $\tilde{\varphi}(\bar{a},b)$ holds for some $b\in \cK^{z}$. Then 
$\tilde{\xi}(\bar{a}_1,\ldots, \bar{a}_n, b)$ holds and by induction we have that $\bar{a}_j=\bd^{k_{j}}(a_j)$ for all $1\leqslant j\leqslant n$. In particular, $\bd^d(a_j)=g_j^d(\bar{a}_1,\ldots, \bar{a}_j)$ for every $d\geqslant 0$ and $1\leqslant j\leqslant n$. As in the case $n=0$, by the axiom scheme (DL) and as in Lemma \ref{lem:envelop}, for any neighbourhood of $\bar{a}_{n+1}$ there is an element $u\in K$ such that 
\[
P(a_1,\ldots, a_n, u)=0 \wedge S(a_1,\ldots, a_n, u)\neq 0\wedge \theta(\bd^m(a_1), \ldots, \bd^m(a_n), \bd^m(u),b). 
\]
Since the topology is Hausdorff and $X_b$ is finite, we must have that $\bar{a}_{n+1}$ is of the form $\bd^{k_{n+1}}(a_{n+1})$. This shows $\varphi(a,b)$ holds. \qed

\section{Applications to dense pairs}\label{sec:app}

The study of pairs of models of a given complete theory is a classical topic in model theory (see \cite{Robinson}, \cite{Macintyre} and \cite{Dries1998} to mention just a few). New developments have enclosed many classical results in different abstract frameworks. Two such frameworks are the theory of lovely pairs of geometric structures developed by A. Berenstein and E. Vassiliev \cite{berenstein-vassiliev2010}, and the theory of dense pairs of theories with existential matroids developed by A. Fornasiero \cite{F}. In this section we will study the theory $T_P$ of dense pairs of models associated to a one-sorted 
open $\cL$-theory of topological fields $T$. Our goal is to show that such theory is closely related with the theory $T_\delta^*$. In Section \ref{sec:DenseP1}, we will recall the definition of the theory $T_P$ and show how it fits into the two above mentioned frameworks. In Section \ref{sec:DenseP2}, we will show how to use $T_{\delta}^*$ to deduce properties of $T_P$. Although most of the results gather in this section concerning $T_P$ are known, the proofs and methods will put in evidence the interesting connexion between the model theory of dense pairs and generic derivations.  

\subsection{Dense pairs of models of $T$}\label{sec:DenseP1}

Given that the literature of the model theory of pairs is quite extensive, we will unify references and cite \cite{F} even if particular cases of cited results where proven before by many different authors. 

Let $T$ be a complete one-sorted geometric $\cL$-theory $T$ extending the theory of fields (not necessarily an open $\cL$-theory of topological fields).\footnote{Fornasiero considers more generally the case where $T$ admits an \emph{existential matroid} (see \cite[Definition 3.25]{F}), but we will not need this level of generality in the present paper.} Given a model $\cM$ of $T$ and a definable subset $X\subseteq M$, we say that $X$ is {\it dense} if $X\cap U\neq \emptyset$ for every $M$-definable subset $U$ of $M$ of $\acl$-dimension 1 (see \cite[Definition 7.1]{F}). Let $\cL_P$ be the language of pairs, that is, defined as $\cL_{P}\coloneqq\cL\cup\{P\}$ for $P$ a new unary predicate. The theory $T_{P}$ of dense pairs of models of $T$ is defined as the $\cL_P$-theory of pairs $(\cK, P(K))$ such that $\cK\models T$, $P(K)$ is $\acl$-closed and dense in $K$ (in the above sense). Equivalently (by \cite[Lemma 7.4]{F}), it corresponds to the $\cL_P$-theory of pairs $(\cK, P(K))$ such that $\cK\models T$, $\langle P(K) \rangle_\cL\preccurlyeq_\cL \cK$ and $P(K)$ is dense in $K$. Among various model-theoretic results which are proven in \cite{F} about the theory $T_P$, what plays a crucial role in this section is the fact that $T_P$ is a complete theory \cite[Theorem 8.3]{F}.

Observe that when $T$ is a one-sorted open $\cL$-theory of topological fields, by Proposition \ref{prop:conseAA1}, $T$ is geometric and therefore, $T_P$ is complete. Note also that in view of Proposition \ref{prop:conseAA2}, the notion of density above defined coincides with the topological notion of density. It is not difficult to see that in this case the theory $T_P$ also coincides with the theory of lovely pairs of geometric theories introduced by A. Berenstein and E. Vassiliev in \cite{berenstein-vassiliev2010}. 

\subsection{Dense pairs and generic derivations}\label{sec:DenseP2}

Throughout this section we suppose $T$ is a one-sorted $\cL$-open theory of topological fields and $\cK$ be a model of $T_\delta^*$. The connection of $T_P$ with the theory $T_\delta^*$ arises via the field of constants $C_K$ of a model $\cK$ of $T_\delta^*$. Recall that by our definition of $T_\delta$, $\Q(\Omega)$ is a subfield of $C_K$. One can readily observe that when $\cK\models \CODF$, the pair $(\cK_{|\cL}, C_K)$ is a dense pair of real-closed fields. The following lemma shows this holds in general for $T_\delta^*$. For the reader's convenience, we give a proof here, following the references \cite[Lemma 7.4]{F}, \cite[Lemma 2.5]{berenstein-vassiliev2010} (see also \cite[Corollary 1.7]{BCPP2018}). 

\begin{lemma}\label{fact:model} The pair $(\cK_{|\cL},C_K)$ is a model of $T_P$ and if $\cK$ is $\vert T\vert^+$-saturated, then $(\cK_{|\cL},C_K)$ is a lovely pair of models of $T$.  
\end{lemma}
\begin{proof}
A direct consequence of the scheme (DL) is that $C_{K}\neq K$. Since $C_{K}$ is topologically dense in $K$ \cite[Lemma 3.12]{guzy-point2010}, $C_{K}$ is dense. So it remains to show that $C_{K}\preccurlyeq_\cL \cK_{|\cL}$. We apply Tarski-Vaught test. Let $\varphi(x,\bar y)$ be an $\cL$-formula and let $\bar b\in C_{K}$. By condition \ref{condA} on $T$, $\varphi(\cK,\bar b)$ is a finite union of finite subsets and open sets. Since $C_{K}$ is algebraically closed in $K$, either $\varphi(\cK,\bar b)\subseteq C_{K}$ or $\varphi(\cK,\bar b)$ contains an open subset. Since $C_{K}$ is topologically dense in $K$, we get the result. 
\end{proof}  

\begin{lemma}\label{lem:elext} For every model $(\cK,F)$ of $T_P$ there is an model $\cK^*$ of $T_\delta^*$ such that $(\cK,F)\preccurlyeq (K^*_{|\cL},C_{K^*})$.  
\end{lemma}

\begin{proof} Since $T_P$ is complete, by Lemma \ref{fact:model}, there is a model $\cK'$ of $T^*_{\delta}$ with constant field $F'$ such that $(\cK_{|\cL}',F')\equiv_{\cL_P} (\cK,F)$. The result follows by Keisler-Shelah's theorem. 
\end{proof}

Let us now show how the previous results allows us to transfer properties of $T_\delta^*$ to $T_P$ using the fact that, for models of $T_\delta^*$, every $\cL_P$-formula defines a set which is $\cL_\delta$-definable by replacing $P(t)$ by the formula $\delta(t)=0$.

\begin{corollary}\label{cor:delta_exist_inf} The theory $T_{P}$ eliminates $\exists^{\infty}$. 
\end{corollary} 
\begin{proof}
This follows directly from the elimination of $\exists^\infty$ in $T_\delta^*$ (Theorem \ref{prop:QE_infty}), and the fact that every $\cL_P$-formula defines a set which is $\cL_\delta$-definable.  
\end{proof} 

\begin{proposition}\label{thm:opencorepairs} The theory $T_P$ has $\cL$-open core.
\end{proposition}

\begin{proof} By Lemma \ref{lem:elext}, let $\cK^*$ be a model of $T_\delta^*$ such that $(\cK_{|\cL}^*,C_{K^*})$ is an $\cL_P$-elementary extension of $(\cK,F)$. Let $\varphi(x,y)$ be an $\cL_P$-formula such that for $a\in \cK^{y}$, $\varphi(\cK,a)$ is open. Then, $\varphi(\cK^*,a)$ is open too. Since $\cK^*$ is a model of $T_\delta^*$ and every $\cL_P$-formula defines a set which is $\cL_\delta$-definable, $\varphi(\cK^*,a)$ is $\cL_\delta$-definable. Now by Theorem \ref{thm:opencoregen}, we have that $\varphi(\cK^*,a)$ is definable by $\psi(x,c)$ where $\psi(x,z)$ is an $\cL$-formula and $c\in (\cK^*)^z$. Then we have that 
\[
(\cK^*_{|\cL},C_{K^*})\models (\forall x)(\varphi(x,a)\leftrightarrow \psi(x,c)), 
\]
and quantifying over $c$, we have that 
\[
(\cK,F) \models (\exists c)(\forall x)(\varphi(x,a)\leftrightarrow \psi(x,c)), 
\]
which shows that $\varphi(\cK,a)$ is $\cL$-definable. 
\end{proof}

As a corollary of the above proposition and Corollary \ref{cor:opencore}, we recover the following result shown by P. Hieronymi and G. Boxall \cite[Corollary 3.4]{Boxall2012} and A. Fornasiero \cite[Theorem 13.11]{F}. 

\begin{corollary}\label{cor:opencorepairs}  Let $T$ be either $\RCF$, $\PCF$, $\RCVF$, $\ACVF_{0,p}$ or the $\cL_\Div$-theory of a henselian valued field of characteristic $0$. Then $T_{P}$ has $\cL$-open core. \qed
\end{corollary}


We finish this section with some remarks on distality. By a result of P. Hieronymi and T. Nell in \cite{HN2017}, the theory of dense pairs of an o-minimal expansion of an ordered group is not distal. A natural question posed by P. Simon asks whether the theory of such pairs always admits a distal expansion \cite[Question 1]{nell2018}. In \cite{nell2018}, T. Nell provided a positive answer to the question for the theory of dense pairs of ordered vector spaces. 

We show in Corollary \ref{cor:dist_expansion} that the theory of dense pairs of real closed fields admits a distal expansion, namely, $\CODF$. We derive analogous results for some theories of dense pairs of henselian valued fields using the following characterization of distality given in \cite{ACGZ}. 

\begin{fact}[{\cite[Main Theorem]{ACGZ}}]\label{fact:dist} Let $K$ be a henselian valued field of characteristic 0 in $\cL_{\Div}$. Then $K$ is distal if and only if it is finitely ramified and both its residue field and value group are distal. \qed
\end{fact} 

\begin{remark}\label{rem:distal} In particular, the following theories are distal: $\PCF$, $\RCVF$ and the $\cL_{\Div}$-theory of $k(\!(t^{\Gamma})\!)$ where $k$ is a field of characteristic $0$ with a distal theory and where $\Gamma$ is an ordered abelian group with a distal theory (for $\PCF$ one uses additionally that adding a small set of constants preserves distality, see \cite{simon2013}). It is worth mentioning that for theories of regular abelian groups, the following properties are equivalent by \cite[Theorem 3.2]{ACGZ}: being distal, being dp-minimal or satisfying the algebraic condition that $\Gamma/p\Gamma$ is finite for every prime $p$. Alternatively, for $\RCVF$ and $\PCF$, distality also follows from dp-minimality.
\end{remark}

\begin{corollary}\label{cor:dist_expansion} If the theory $T$ is distal, then $T_\delta^*$ is a distal expansion of $T_P$. In particular, $T_P$ admits a distal expansion when $T$ is either $\RCF$, or $\PCF$, or $\RCVF$ or the $\cL_{\Div}$-theory of a finitely ramified henselian valued field with distal residue field and value group (e.g., $\R(\!(t)\!)$, or $\Q_p(\!(t)\!)$).
\end{corollary} 

\begin{proof} Since distality transfers from $T$ to $T_\delta^*$ (Theorem \ref{thm:distality}), Lemma \ref{lem:elext} implies that $T_\delta^*$ is a distal expansion  of $T_P$ whenever $T$ is distal. Thus,  for $\RCF$ the result follows from the well-known fact that any o-minimal theory is distal. The remaining cases follow from Remark \ref{rem:distal}. 
\end{proof}

It is worthy to note that A. Fornasiero and E. Kaplan \cite{fornasiero-kaplan} generalized the previous result for some o-minimal expansions of $\RCF$. 
  
\medskip

As a consequence of results of A. Chernikov and S. Starchenko in \cite{chernikov-star2018}, definable relations in models of $T_P$ satisfied the so called strong Erd\H{o}s-Hajnal property  (see \cite[Definition 1.6, Theorem 6.10 (3)]{chernikov-star2018}).

\begin{corollary}\label{cor:distality_erdos} If $T$ is distal, then definable relations in models of $T_P$ satisfy the strong Erd\H{o}s-Hajnal property. This holds in particular for models of $T_P$ when $T$ is $\RCF$, $\PCF$ or $\RCVF$, or when $T$ is the $\cL_{\Div}$-theory of a finitely ramified henselian valued field with distal residue field and value group.
\end{corollary} 
\begin{proof}
This follows from Corollary \ref{cor:dist_expansion} and \cite[Corollary 4.8]{chernikov-star2018}.
\end{proof}

\appendix

\section{Classical transfers}\label{appendix}

Through this section we let $T$ be an open $\cL$-theory of topological fields. Let $\bU$ be a monster model of $T_\delta^*$ and $A$ be some small subset. 

\begin{lemma}\label{lem:types} 
Let $x$ be a tuple of $\bF$-variables and let $z$ be a tuple of auxiliary sort variables. Then for $a\in \bU^x$ and $e\in \bU^z$, the $\cL_\delta$-type $tp_\delta(a,e/A)$ is determined by the infinite sequence of $\cL$-types $\{tp(\bd^m(a), e / \langle A\rangle_{\cL_\delta}): m\in \N\}$. 
\end{lemma}

\begin{proof} This follows from Corollary \ref{cor:eq}.
\end{proof}

\begin{corollary}\label{cor:indiscernibles} Let $x$ and $z$ be as in the previous lemma. Let $(a_i, e_i)_{i\in I}$ be a sequence where $a_i\in \bU^x$ and $e_i\in \bU^z$. Then the sequence $(a_i, e_i)_{i\in I}$ is $\cL_\delta$-indiscernible over $A$ if and only if for each $m\in \N$, the sequence $(\bd^m(a_i), e_i)_{i\in I}$ is $\cL$-indiscernible over $\langle A\rangle_{\cL_\delta}$. \qed
\end{corollary} 

\begin{theorem}\label{thm:ANIP}
$T$ is $\NIP$ if and only if $T_\delta^*$ is $\NIP$. 
\end{theorem}

\begin{proof} Suppose $T_\delta^*$ is not $\NIP$. Let $\varphi(x,z)$ be an $\cL_\delta$-formula with $x$ a tuple of $\bF$-variables and $z$ a tuple of auxiliary sort variables and suppose $\varphi$ has $\mathrm{IP}$. By Corollary \ref{cor:eq}, $\varphi$ is equivalent to $\tilde{\varphi}(\bd^m(x),z)$, where $\tilde{\varphi}(\bar{x},z)$ is an $\cL$-formula. Then, since $\varphi$ has $\mathrm{IP}$ so does the formula $\tilde{\varphi}$. The converse is clear, since being $\NIP$ is preserved by reducts (see Remark \ref{rem:consistency}). 
\end{proof}

To show the transfer of distality, we use the following definition which appears in \cite{HN2017}, which is slightly different from the original definition given by P. Simon in \cite{simon2013}. In the following definition we let $\cL$ be any first order language, $T$ be a complete $\cL$-theory and $\bU$ be a monster model of $T$. 

\begin{definition}[{\cite[Definition 1.3]{HN2017}}] Let $\varphi(x_1,\ldots,x_n;y)$ be a partitioned $\cL$-formula, where $x_i$, $1\leqslant i\leqslant n$ is a $p$-tuple of variables and $y$ is a $q$-tuple of variables, $p, q>0$. Then $\varphi$ is distal (in $T$) if for every $b\in \bU^q$, and every indiscernible sequence $(a_i)_{i\in I}$ in $\bU^p$ such that 
\begin{enumerate}
\item $I=I_1+c+I_2$, where both $I_1, I_2$ are (countable) infinite dense linear orders without end points and $c$ is a single element with $I_1<c<I_2$,
\item the sequence $(a_i)_{i\in I_1+I_2}$ in $\bU^p$ is indiscernible over $b$,
\end{enumerate}
then $\bU\models \varphi(a_{i_1},\ldots,a_{i_n};b)\leftrightarrow \varphi(a_{j_1},\ldots,a_{j_n};b)$ with $i_1<\ldots<i_n$, $j_1<\ldots<j_n$ in $I$. 

A theory $T$ is distal if every formula is distal in $T$. 
\end{definition}

The transfer of distality from $T$ to $T_\delta^*$ was already noted for some theories $T$ by A. Chernikov, which by now have been included in \cite{ACGZ} (where they consider more generally the transfer of distality for fields endowed with several operators). The converse has not been, to our knowledge, observed before. Note that since distality is not preserved under reducts, the converse implication is not straightforward as in Theorem \ref{thm:ANIP}. 

\begin{theorem}\label{thm:distality} 
$T$ is distal if and only if $T_\delta^*$ is distal.
\end{theorem}
\begin{proof}
Let us check that in $T_\delta^*$ every formula is distal. Let $\varphi(x_1,z_1,\ldots,x_n,z_n;y,w)$ be a partitioned $\cL_{\delta}$-formula where each $x_i$ is a $p$-tuple of $\bF$-variables, each $z_i$ is a $q$-tuple of variables of fixed auxiliary sorts $\bS_1,\ldots, \bS_q$, $y$ is a tuple of $\bF$-variables and $w$ is a tuple of auxiliary sorts. Let $\bU$ be a monster model of $T_\delta^*$. By Corollary \ref{cor:eq}, there is an $\cL$-formula $\tilde{\varphi}(\bar{x}_1,z_1, \ldots, \bar{x}_n, z_n; \bar{y},w)$ such that $\varphi(x_1,z_1,\ldots,x_n,z_n;y,w)$ is equivalent to $\tilde{\varphi}(\bd^m(x_1),z_1,\ldots,\bd^m(x_n),z_n;\bd^m(y),w)$ for some $m\geqslant 0$.
Take an $\cL_{\delta}$-indiscernible sequence $(a_i, e_i)_{i\in I}$ in $\bU$ where $(a_i,e_i)\in \bU^{(x_1,z_1)}$ and $I=I_1+c+I_2$ with $I_1 , I_2$ infinite dense linear orders without end points. Let $(b,d)$ be a tuple in $\bU^{(y,w)}$, and assume that $(a_i,e_i)_{i\in I_1+I_2}$ is $\cL_{\delta}$-indiscernible over $(b,d)$. Then, by Corollary \ref{cor:indiscernibles}, the sequence $(\bar{\delta}^{m}(a_i), e_i)_{i\in I}$ (resp. $(\bar{\delta}^{m}(a_i),e_i)_{i\in I_1+I_2}$) is $\cL$-indiscernible (resp. $\cL$-indiscernible over $B$ where $B=\{\delta^m(b):m\in \N\}\cup \{d\})$) for every $m\in \N$. Since $T$ is distal, the $\cL$-formula $\tilde{\varphi}$
is distal, which implies the distality of $\varphi$. 

For the converse, suppose $\varphi(x_1,z_1,\ldots, x_n,z_n; y, w)$ is an $\cL$-formula which is not distal in $T$. Consider the $\cL_\delta$-formula $\psi$
\[
\varphi(x_1,z_1,\ldots, x_n,z_n; y, w) \wedge \bigwedge_{j=1}^n \delta(x_i)=0 \wedge \delta(y)=0. 
\]
Let $A\subseteq \bU$ be such that all elements of the field sort are in the constant field $C_\bU$ of $\bU$. Let $(a_i,e_i)_{i\in I}$ be an $\cL$-indiscernible sequence over $A$, where $a_i\in \bU^{x_1}$ and $e_i\in \bU^{z_1}$. If $a_i\in C_\bU$ for each $i\in I$, then by Corollary \ref{cor:indiscernibles}, we have that $(a_i,e_i)_{i\in I}$ is also $\cL_\delta$-indiscernible over $A$. Then if $(a_i, e_i)_{i\in I}$ and $(b,d)\in \bU^{(y,w)}$ are a counterexample for the distality of $\varphi$ in $T$, the same witnesses show that $\psi$ is not distal in $T_\delta^*$. 
\end{proof} 

We finish by showing that $T_\delta^*$ eliminates the field quantifier $\exists^\infty$. It was proven for $\CODF$ by the second author in \cite{point2011} and played a crucial role in the (first) proof that $\CODF$ has open core. The result is to our knowledge new in the general setting.

\begin{theorem}\label{prop:QE_infty} The theory $T_\delta^*$ eliminates the field sort quantifier $\exists^\infty$.
\end{theorem}

\begin{proof}
For simplicity, suppose $\cL$ is one-sorted. Let $\cK$ be a model of $T_{\delta}^*$. Let $X\subseteq K^{n+1}$ be an $\cL_{\delta}$-definable set of order $\bar{m}=(m_1,\ldots, m_{n+1})$. For $m=\max_i\{m_i\}$, let $Y\subseteq K^{\bar{d}+1}$ be the $\cL$-definable set given in Lemma \ref{lem:envelop}. Since $T$ eliminates $\exists^{\infty}$, there is a finite bound $n_{Y}$ such that for any tuple $\bar e\in K^{n(d+1)}$, either $Y_{\bar e}$ is infinite or has cardinality $\leqslant n_{Y}$. By Lemma \ref{lem:envelop}, if $X_a$ is finite for $a\in K^n$, then $\vert X_{a}\vert=\vert Y_{\bd^d(a)}\vert\leqslant n_Y$, so the same bound shows the result for $X$. 

The proof in the multi-sorted case is essentially the same but needs a version of Lemma \ref{lem:envelop} where we also allow auxiliary-sort parameter variables. Although syntactically more tedious, the proof of such a lemma is however the same. 
\end{proof}

\subsection*{Acknowledgements} We would like to thank: Marcus Tressl for encouraging discussions and for sharing with us his proof strategy to show elimination of imaginaries in $\CODF$; Arno Fehm and Philip Dittmann for interesting discussions around henselian valued fields; Silvain Rideau for helpful comments on an earlier version of this paper; the Institute of Algebra of the Technische Universit\"at Dresden for its hospitality during a visit of the second author in May 2019; and l'\'Equipe de logique math\'ematique de l'IMJ-PRG for its hospitality during a visit of the first author in June 2021. The first author was partially funded by the ERC project TOSSIBERG (Grant Agreement 637027) and the Deutsche Forschungsgemeinschaft (DFG) - 404427454. Finally, we would like to point out that a previous version of the present work followed a slightly different approach to show the open core property which can also be of interest (see arXiv:1912.0791v1 [math.LO]).

\bibliographystyle{siam}
\bibliography{biblio}

\begin{thebibliography}{10}

\bibitem{ACGZ}
{\sc M.~Aschenbrenner, A.~Chernikov, A.~Gehret, and M.~Ziegler}, {\em Distality
  in valued fields and related structure}.
\newblock arXiv:2008.09889.

\bibitem{ADH}
{\sc M.~Aschenbrenner, L.~van~den Dries, and J.~van~der Hoeven}, {\em
  Asymptotic differential algebra and model theory of transseries}, vol.~195 of
  Annals of Mathematics Studies, Princeton University Press, Princeton, NJ,
  2017.

\bibitem{berenstein-vassiliev2010}
{\sc A.~Berenstein and E.~Vassiliev}, {\em On lovely pairs of geometric
  structures}, Ann. Pure Appl. Logic, 161 (2010), pp.~866--878.

\bibitem{Boxall2012}
{\sc G.~Boxall and P.~Hieronymi}, {\em Expansions which introduce no new open
  sets}, J. Symbolic Logic, 77 (2012), pp.~111--121.

\bibitem{BMR}
{\sc T.~Brihaye, C.~Michaux, and C.~Rivi\`ere}, {\em Cell decomposition and
  dimension function in the theory of closed ordered differential fields}, Ann.
  Pure Appl. Logic, 159 (2009), pp.~111--128.

\bibitem{brouette2015}
{\sc Q.~Brouette}, {\em Differential algebra, ordered fields and model theory},
  PhD thesis, Universit\'e de Mons, 2015.

\bibitem{BCPP2018}
{\sc Q.~Brouette, G.~Cousins, A.~Pillay, and F.~Point}, {\em Embedded
  {P}icard-{V}essiot extensions}, Comm. Algebra, 46 (2018), pp.~4609--4615.

\bibitem{BCP}
{\sc Q.~Brouette, P.~Cubides~Kovacsics, and F.~Point}, {\em Strong density of
  definable types and closed ordered differential fields}, The Journal of
  Symbolic Logic,  (2019), pp.~1--20.

\bibitem{chernikov_simon}
{\sc A.~Chernikov and P.~Simon}, {\em Henselian valued fields and
  inp-minimality}, The Journal of Symbolic Logic,  (2019), pp.~1--17.

\bibitem{chernikov-star2018}
{\sc A.~Chernikov and S.~Starchenko}, {\em Regularity lemma for distal
  structures}, J. Eur. Math. Soc. (JEMS), 20 (2018), pp.~2437--2466.

\bibitem{delon81}
{\sc F.~Delon}, {\em Types sur $\mathbb {C} ((x))$}, Groupe d'\'etude de
  th\'eories stables, 2 (1978-1979).

\bibitem{flenner}
{\sc J.~Flenner}, {\em Relative decidability and definability in {H}enselian
  valued fields}, J. Symbolic Logic, 76 (2011), pp.~1240--1260.

\bibitem{F}
{\sc A.~Fornasiero}, {\em Dimensions, matroids, and dense pairs of first-order
  structures}, Ann. Pure Appl. Logic, 162 (2011), pp.~514--543.

\bibitem{fornasiero-kaplan}
{\sc A.~Fornasiero and E.~Kaplan}, {\em Generic derivations on o-minimal
  structures}, Journal of Mathematical Logic,  (2020), p.~2150007.
\newblock online ready.

\bibitem{gurevich_schmitt}
{\sc Y.~Gurevich and P.~H. Schmitt}, {\em The theory of ordered abelian groups
  does not have the independence property}, Transactions of the American
  Mathematical Society, 284 (1984), pp.~171--182.

\bibitem{guzy-point2010}
{\sc N.~Guzy and F.~Point}, {\em Topological differential fields}, Annals of
  Pure and Applied Logic, 161 (2010), pp.~570 -- 598.

\bibitem{GP12}
{\sc N.~Guzy and F.~Point}, {\em Topological differential fields and dimension
  functions}, J. Symbolic Logic, 77 (2012), pp.~1147--1164.

\bibitem{HHM2006}
{\sc D.~Haskell, E.~Hrushovski, and D.~Macpherson}, {\em Definable sets in
  algebraically closed valued fields: elimination of imaginaries}, J. Reine
  Angew. Math., 597 (2006), pp.~175--236.

\bibitem{HN2017}
{\sc P.~Hieronymi and T.~Nell}, {\em Distal and non-distal pairs}, J. Symb.
  Log., 82 (2017), pp.~375--383.

\bibitem{HMR2018}
{\sc E.~Hrushovski, B.~Martin, and S.~Rideau}, {\em Definable equivalence
  relations and zeta functions of groups}, J. Eur. Math. Soc. (JEMS), 20
  (2018), pp.~2467--2537.
\newblock With an appendix by Raf Cluckers.

\bibitem{jahnke_simon_walsberg_2017}
{\sc F.~Jahnke, P.~Simon, and E.~Walsberg}, {\em Dp-minimal valued fields}, The
  Journal of Symbolic Logic, 82 (2017), p.~151–165.

\bibitem{johnson18}
{\sc W.~Johnson}, {\em The canonical topology on dp-minimal fields}, Journal of
  Mathematical Logic, 18 (2018), p.~1850007 (23).

\bibitem{lang}
{\sc S.~Lang}, {\em Algebra}, vol.~211 of Graduate Texts in Mathematics,
  Springer-Verlag, New York, third~ed., 2002.

\bibitem{Macintyre}
{\sc A.~Macintyre}, {\em Dense embeddings. {I}. {A} theorem of {R}obinson in a
  general setting}, Springer, Berlin, 1975, pp.~200--219. Lecture Notes in
  Math., Vol. 498.

\bibitem{marker1996}
{\sc D.~Marker}, {\em Chapter 2: Model Theory of Differential Fields}, vol.~5
  of Lecture Notes in Logic, Springer-Verlag, Berlin, 1996, pp.~38--113.

\bibitem{M}
{\sc L.~Mathews}, {\em Cell decomposition and dimension functions in
  first-order topological structures}, Proc. London Math. Soc. (3), 70 (1995),
  pp.~1--32.

\bibitem{mellor2006}
{\sc T.~Mellor}, {\em Imaginaries in real closed valued fields}, Ann. Pure
  Appl. Logic, 139 (2006), pp.~230--279.

\bibitem{nell2018}
{\sc T.~Nell}, {\em Distal and non-distal behavior in pairs}, Mathematical
  Logic Quarterly, 65 (2019), pp.~23--36.

\bibitem{pillay87}
{\sc A.~Pillay}, {\em First order topological structures and theories}, The
  Journal of Symbolic Logic, 52 (1987), pp.~763--778.

\bibitem{point2011}
{\sc F.~Point}, {\em Ensembles d\'efinissables dans les corps ordonn\'es
  diff\'erentiellement clos}, C. R. Math. Acad. Sci. Paris, 349 (2011),
  pp.~929--933.

\bibitem{Fr}
\leavevmode\vrule height 2pt depth -1.6pt width 23pt, {\em Definability of
  types and vc density in differential topological fields}, Archive for
  Mathematical Logic, 57 (2018), pp.~809--828.

\bibitem{pre-ro-84}
{\sc A.~Prestel and P.~Roquette}, {\em Formally $p$-adic fields}, Lecture Notes
  in Mathematics, Springer-Verlag, Berlin, 1984.

\bibitem{PZ}
{\sc A.~Prestel and M.~Ziegler}, {\em Model-theoretic methods in the theory of
  topological fields}, J. Reine Angew. Math., 299(300) (1978), pp.~318--341.

\bibitem{rideau}
{\sc S.~Rideau}, {\em Imaginaries and invariant types in existentially closed
  valued differential fields}, J. Reine Angew. Math., 750 (2019), pp.~157--196.

\bibitem{Robinson}
{\sc A.~Robinson}, {\em Solution of a problem of {T}arski}, Fund. Math., 47
  (1959), pp.~179--204.

\bibitem{R1980}
{\sc M.~Rosenlicht}, {\em Differential valuations}, Pacific J. Math., 86
  (1980), pp.~301--319.

\bibitem{Scanlon}
{\sc T.~Scanlon}, {\em A model complete theory of valued {$D$}-fields}, J.
  Symbolic Logic, 65 (2000), pp.~1758--1784.

\bibitem{simon2013}
{\sc P.~Simon}, {\em Distal and non-distal {NIP} theories}, Ann. Pure Appl.
  Logic, 164 (2013), pp.~294--318.

\bibitem{simon-nip}
\leavevmode\vrule height 2pt depth -1.6pt width 23pt, {\em A guide to {NIP}
  theories}, vol.~44 of Lecture Notes in Logic, Association for Symbolic Logic,
  Chicago, IL; Cambridge Scientific Publishers, Cambridge, 2015.

\bibitem{simon-walsberg2016}
{\sc P.~Simon and E.~Walsberg}, {\em Tame {T}opology over dp-{M}inimal
  {S}tructures}, Notre Dame J. Form. Log., 60 (2019), pp.~61--76.

\bibitem{singer1978}
{\sc M.~F. Singer}, {\em The model theory of ordered differential fields}, J.
  Symbolic Logic, 43 (1978), pp.~82--91.

\bibitem{solanki}
{\sc N.~J. Solanki}, {\em Uniform companions for relational expansions of large
  differential fields in characteristic 0}, PhD thesis, Manchester University,
  2014.

\bibitem{tent-ziegler2012}
{\sc K.~Tent and M.~Ziegler}, {\em A course in model theory}, vol.~40 of
  Lecture Notes in Logic, Association for Symbolic Logic, La Jolla, CA;
  Cambridge University Press, Cambridge, 2012.

\bibitem{Tressl}
{\sc M.~Tressl}, {\em The uniform companion for large differential fields of
  characteristic 0}, Trans. Amer. Math. Soc., 357 (2005), pp.~3933--3951.

\bibitem{vandendries1989}
{\sc L.~van~den Dries}, {\em Dimension of definable sets, algebraic boundedness
  and henselian fields}, Annals of Pure and Applied Logic, 45 (1989), pp.~189
  -- 209.

\bibitem{Dries1998}
\leavevmode\vrule height 2pt depth -1.6pt width 23pt, {\em Dense pairs of
  o-minimal structures}, Fund. Math., 157 (1998), pp.~61--78.

\end{thebibliography}

\end{document}